\documentclass[11pt,a4paper]{amsart}
\usepackage[margin=1in]{geometry}
\usepackage[utf8]{inputenc}
\usepackage{amsmath}
\usepackage{amsfonts}
\usepackage{xypic}
\usepackage{amsthm}
\usepackage{mathrsfs}
\usepackage{hyperref}
\usepackage{amssymb}
\usepackage[all]{xy}
\usepackage{color}

\numberwithin{equation}{section}

\newtheorem{theorem}{Theorem}[section]
\newtheorem{lemma}[theorem]{Lemma}
\newtheorem{proposition}[theorem]{Proposition}
\newtheorem{corollary}[theorem]{Corollary}

\newtheorem{proposition-definition}[theorem]{Proposition-definition}

\theoremstyle{remark}
\newtheorem{remark}[theorem]{Remark}

\DeclareMathOperator{\Jac}{Jac}

\DeclareMathOperator{\Blow}{Blow}

\renewcommand{\det}{\mathrm{det}}
\newcommand{\st}{\mathrm{st}}
\newcommand{\ch}{\mathrm{ch}}
\newcommand{\Todd}{\mathrm{Todd}}

\newcommand{\Pic}{\mathrm{Pic}}
\newcommand{\Hilb}{\mathrm{Hilb}}
\newcommand{\supp}{\mathrm{supp}}
\newcommand{\End}{\mathrm{End}}
\newcommand{\Ext}{\mathrm{Ext}}

\newcommand{\Sym}{\mathrm{Sym}}

\newcommand{\id}{\mathbf{1}}
\newcommand{\rk}{\mathrm{rk}}

\newcommand{\SL}{\mathrm{SL}}

\newcommand{\U}{\mathrm{U}}

\newcommand{\image}{\mathrm{Im}}

\newcommand{\Tot}{\mathrm{Tot}}
\newcommand{\tr}{\mathrm{tr} \,}

\newcommand{\Fix}{\mathrm{Fix}}

\newcommand{\Prym}{\mathrm{Prym}}

\newcommand{\sm}{\mathrm{sm}}
\newcommand{\ssl}{\mathrm{ssl}}
\newcommand{\Nm}{\mathrm{Nm}}
\newcommand{\Sp}{\mathrm{Sp}}
\newcommand{\PGL}{\mathrm{PGL}}

\newcommand{\Aa}{\mathcal{A}}
\newcommand{\Bb}{\mathcal{B}}

\newcommand{\Dd}{\mathcal{D}}
\newcommand{\Ee}{\mathcal{E}}
\newcommand{\Ff}{\mathcal{F}}
\newcommand{\Gg}{\mathcal{G}}
\newcommand{\Ii}{\mathcal{I}}
\newcommand{\Hh}{\mathcal{H}}

\newcommand{\Ll}{\mathcal{L}}
\newcommand{\Mm}{\mathcal{M}}
\newcommand{\Nn}{\mathcal{N}}
\newcommand{\Oo}{\mathcal{O}}
\newcommand{\Pp}{\mathcal{P}}

\newcommand{\Tt}{\mathcal{T}}
\newcommand{\Uu}{\mathcal{U}}

\newcommand{\h}{\mathrm{h}}
\newcommand{\q}{\mathrm{q}}
\renewcommand{\v}{\mathrm{v}}
\newcommand{\w}{\mathrm{w}}

\renewcommand{\H}{\mathrm{H}}
\newcommand{\I}{\mathrm{I}}
\newcommand{\J}{\mathrm{J}}
\newcommand{\M}{\mathrm{M}}
\newcommand{\N}{\mathrm{N}}

\renewcommand{\P}{\mathrm{P}}

\newcommand{\ol}[1]{\overline{#1}}

\newcommand{\cC}{\mathbf{C}}
\newcommand{\dD}{\mathbf{D}}
\newcommand{\hH}{\mathbf{H}}
\newcommand{\jJ}{\mathbf{J}}
\newcommand{\iI}{\mathbf{I}}
\newcommand{\mM}{\mathbf{M}}
\newcommand{\nN}{\mathbf{N}}
\newcommand{\pP}{\mathbf{P}}

\newcommand{\sS}{\mathbf{S}}
\newcommand{\tT}{\mathbf{T}}
\newcommand{\vV}{\mathbf{v}}
\newcommand{\wW}{\mathbf{w}}
\newcommand{\xX}{\mathbf{X}}
\newcommand{\yY}{\mathbf{Y}}

\renewcommand{\AA}{\mathbb{A}}

\newcommand{\CC}{\mathbb{C}}

\newcommand{\ZZ}{\mathbb{Z}}
\newcommand{\LL}{\mathbb{L}}

\newcommand{\KK}{\mathbb{K}}

\newcommand{\PP}{\mathbb{P}}

\newcommand{\WW}{\mathbb{W}}

\newcommand{\Fff}{\mathscr{F}}
\newcommand{\Ggg}{\mathscr{G}}

\newcommand{\BBB}{\mathrm{(BBB)}}
\newcommand{\BAA}{\mathrm{(BAA)}}
\newcommand{\ABA}{\mathrm{(ABA)}}
\newcommand{\AAB}{\mathrm{(AAB)}}

\newcommand{\quotient}[2]{{\raisebox{.2em}{\thinspace $#1$ \hspace{-0.2cm}}\left / \raisebox{-.15em}{\hspace{-0.1cm} $#2$}\right.}}

\newcommand\Quotient[2]{
        \mathchoice
            {
                \text{\raise1ex\hbox{\thinspace $#1$}\Big{/} \lower1ex\hbox{$#2$} \thinspace}%
            }
            {
                #1\,/\,#2
            }
            {
                #1\,/\,#2
            }
            {
                #1\,/\,#2
            }
    }

\newcommand\GIT[2]{
        \mathchoice
            {
                \text{\raise1ex\hbox{\thinspace $#1$}\Big{/}\!\!\!\!\Big{/} \lower1ex\hbox{$#2$} \thinspace}%
            }
            {
                #1\,/\,#2
            }
            {
                #1\,/\,#2
            }
            {
                #1\,/\,#2
     a       }
    }

\newcommand{\morph}[6]{\begin{array}{cccc} #6: & #1  & \stackrel{#5}{\longrightarrow} &  #2  \\ & #3 & \longmapsto & #4  \end{array}}

\newcommand{\birrat}[6]{\begin{array}{cccc} #6: & #1  & \stackrel{#5}{\dashrightarrow} &  #2  \\ & #3 & \longmapsto & #4  \end{array}}

\title[Degeneration of natural Lagrangians and Prymian integrable systems]{\bf Degeneration of natural Lagrangians and Prymian integrable systems}

\author[E. Franco]{Emilio Franco}
\address{Emilio Franco,
\newline\indent Centro de An\'alise Matem\'atica, Geometria e Sistemas Din\^{a}micos, 
\newline\indent Instituto Superior T\'ecnico, Universidade de Lisboa, 
\newline\indent Av. Rovisco Pais s/n, 1049-001 Lisboa, Portugal}
\email{emilio.franco@tecnico.ulisboa.pt}

\thanks{EF holds an FCT Investigator grant supported by the Scientific Employment Stimulus program, fellowship reference CEECIND/04153/2017, funded by FCT (Portugal).}

\date{\today}

\begin{document}

\begin{abstract}
Starting from an anti-symplectic involution on a K3 surface, one can consider a natural Lagrangian subvariety inside the moduli space of sheaves over the K3. One can also construct a Prymian integrable system following a construction of Markushevich--Tikhomirov, extended by Arbarello--Sacc\`a--Ferretti, Matteini and Sawon--Chen. In this article we address a question of Sawon, showing that these integrable systems and their associated natural Lagrangians degenerate, respectively, into fix loci of involutions considered by Heller--Schaposnik, Garcia-Prada--Wilkins and Basu--Garcia-Prada.

Along the way we find interesting results such as the proof that the Donagi--Ein--Lazarsfeled degeneration is a degeneration of symplectic varieties, a generalization of this degeneration, originally described for K3 surfaces, to the case of an arbitrary smooth projective surface, and a description of the behaviour of certain involutions under this degeneration. 
\end{abstract}

\maketitle

\tableofcontents

\section{Introduction}

\subsection{Context and motivation}

By means of non-abelian Hodge theory \cite{hitchin-self, simpson1, simpson2, donaldson, corlette}, the moduli space of Higgs bundles carries a hyperK\"ahler structure which naturally defines a triple of symplectic structures on it, each holomorphic with respect to one of the three complex structures. For one of this complex structures, the Higgs moduli space is a quasiprojective variety, further equipped with a proper fibration onto a vector space whose generic fibres are abelian Lagrangians with respect to the corresponding holomorphic symplectic form. These data define the Hitchin integrable system \cite{hitchin_duke}. 

The moduli space of pure dimension $1$ sheaves on a symplectic surface ({\it i.e.} K3 or abelian) can be equipped with the Mukai \cite{mukai1} holomorphic symplectic form and with the Le Potier (support) fibration \cite{lepotier}, whose generic fibres are Jacobians. After Beauville \cite{beauville_fibr}, these fibres are also Lagrangians. On the case of K3 surfaces, the base of the Le Potier morphism is a linear system, hence the moduli space of pure dimension $1$ sheaves on a K3 becomes a projective integrable system, named the Beauville--Mukai integrable system. Putting aside that it is projective, the Beauville--Mukai integrable system has a similar description to that of the Hitchin system obtained via the spectral correspondence. These similarities become even more explicit with the construction of the Donagi--Ein--Lazarsfeld \cite{donagi&ein&lazarsfeld} degeneration of the first integrable system into the later. 

The cohomological structure of the Hitchin system is very rich and many surprising identities occur within this framework, giving rise to a great number of conjectures. One of them is topological Mirror symmetry conjecture \cite{hausel&thaddeus}, predicting the equality between the stringy E-polynomials of Higgs moduli spaces for a pair of Langlands dual groups, proven for $\SL(n,\CC)$ and $\PGL(n,\CC)$ \cite{hausel&thaddeus, groechenig&weyss&ziegler}. Another conjectural cohomological identity is the P=W conjecture \cite{deCataldo&hausel&migliorini}, predicting that the morphism in cohomology induced by the non-abelian Hodge correspondence exchanges the preverse filtration associated to the Hitchin fibration with the weight filtration on the associated character variety. This was proven by de Cataldo--Hausel--Migliorini \cite{deCataldo&hausel&migliorini} in the rank $2$ case, and by de Cataldo--Maulik--Shen \cite{deCataldo&maulik&shen_1, deCataldo&maulik&shen_2} in the case of base curves of genus $2$. The Donagi--Ein--Lazarsfeld degeneration in the case of abelian surfaces was a key element in the work of de Cataldo--Maulik--Shen. Using this degeneration, they constructed a specialization morphism in cohomology, and, so, they could apply the powerful machinery developped by Markman \cite{markman_2, markman_3, markman_4} for the study the cohomology of the moduli of sheaves on a K3, to the Higgs moduli space.

The cohomological $\chi$-independence is another astonishing property of the moduli spaces of (twisted) Higgs bundles and pure dimension $1$ sheaves on del Pezzo surfaces. It states that the intersection cohomology is independent from the Euler characteristic $\chi$ of the classified objects and was recently proven by Maulik--Shen \cite{maulik&shen}.

Making use of the hyperK\"ahler structure, \cite{kapustin&witten} Kapustin and Witten introduced branes in the Higgs moduli space, setting that a $\BBB$-brane is a hyperholomorphic subvariety supporting a hyperholomorphic sheaf, while a $\BAA$-brane is a flat bundle over a complex Lagrangian subvariety for the holomorphic symplectic form associated to the first complex structure. Substituting the first by the second and third complex structures in this definition, we obtain $\ABA$ and $\AAB$-branes, respectively. As indicated in \cite{kapustin&witten}, the Mirror symmetry conjecture predicts a duality between $\BBB$ and $\BAA$-branes. This conjecture has motivated many authors to construct and study branes on Higgs moduli spaces. 
Most of the previous constructions are obtained by considering fixed loci of involutions on the moduli space, we highlight the early construction of $\ABA$-branes by Baraglia and Schaposnik \cite{baraglia&schaposnik} who considered those involutions obtained out of a pair of anti-holomorphic involutions on the base curve and on the group, the $\BBB$ and $\BAA$-branes obtained from natural involutions considered by Heller--Schaposnik \cite{heller&schaposnik} and Garcia-Prada--Wilkins \cite{garciaprada&wilkins} out of holomorphic involutions on the base curve, the work of Garcia-Prada and Ramanan \cite{garciaprada&ramanan} who classified the involutions obtained out of a combination of outer automorphisms of the group and tensorization, and the recent work \cite{basu&garciaprada}, where Basu and Garcia-Prada extended this study to include the action of holomorphic involution of the base curve, a set-up who already appeared in \cite{BCFG} for the case of elliptic curves. As indicated by Gukov in \cite{gukov}, Mirror duality between  $\BBB$ and $\BAA$-branes would imply certain cohomological relations between their support, a direction that was taken by Hausel--Mellit--Pei \cite{hausel&mellit&pei} to provide strong evidence for the pair of branes considered by Hitchin in \cite{hitchin_char} which are constructed out of the pair of Nadler--Langlads groups \cite{nadler} given by $\Sp(2m,\CC)$ and $\U(m,m)$.

Kapustin--Witten's definition of branes extends naturally to hyperK\"ahler varieties other than the Higgs moduli space. The case of the moduli space of sheaves over a symplectic surface was considered by the author, Jardim and Menet in \cite{GYM}, where they constructed branes of any type arising as fixed loci of natural involutions on the moduli induced by involutions on the surface, and studied the behaviour of these natural branes under some correspondences. 

Within this setting, the natural involution associated to an anti-symplectic involution on a K3 surface is again anti-symplectic, and its fixed locus defines a complex Lagrangian subvariety of the moduli, which we call natural Lagrangian subvariety. One obtains a symplectic involution by composing this natural anti-symplectic involution with the dualizing involution on the fibres (perhaps tensoring with a line bundle), which is also anti-symplectic. Hence, considering the fixed locus of the symplectic involution constructed out of anti-symplectic involution on a K3 surface, one obtains a class of integrable systems whose Lagrangian fibres are Prym varieties. These Prymian integrable systems on K3 surfaces were first considered by Markushevich--Tikhomirov \cite{markushevich&tikhomirov}, and extended by Arbarello--Sacc\`a--Ferretti \cite{ASF}, Matteini \cite{matteini} and Sawon--Shen \cite{sawon&shen, shen}. In \cite{sawon}, Sawon conjectured that these Prymian integrable systems degenerate into integrable systems related to the Hitchin system, leaving open the description of these conjectural systems. 

Sawon's conjecture was the motivation for our work as such degeneration of Prymian integrable system and that of the associated natural Lagrangian subverieties, could open the door for a cohomological study of pairs of $\BBB$ and $\BAA$-branes in the Hitchin system, leading perhaps to strong evidence of their duality, by means of the specialization morphism in cohomology given in \cite{deCataldo&maulik&shen_1, deCataldo&maulik&shen_2}. 

In their recent paper \cite{sawon&shen}, Sawon and Shen provided a degeneration of a particular choice of Prymian integrable system into the $\Sp(2m,\CC)$-Higgs moduli space.

\subsection{Main results}

In this article we extend the Donagi--Ein--Lazarsfeld construction \cite{donagi&ein&lazarsfeld} to the case of a curve fitting in an arbitrary smooth projective surface, obtaining a degeneration of the moduli of pure dimension $1$ sheaves on the surface into the moduli space of Higgs bundles on the curve, twisted by the normal bundle of the curve inside the surface [Theorem \ref{tm Mm_S}].We highlight that, in particular, our degeneration connects [Remark \ref{rm chi-independence}] the two moduli spaces for which the cohomological $\chi$-independence is known to hold \cite{maulik&shen}. We also study how certain involutions fit into this degeneration, finding that natural involutions and their composition with the dualizing involution on the moduli space of pure dimension $1$ sheaves on the surface, degenerate into involutions on the moduli space of twisted Higgs bundles that we describe explicitely [Lemmas \ref{lm eta_sS at t=0} and \ref{lm fibrewise description of lambda}]. When the twist is the canonical bundle of the curve, {\it i.e.} for Higgs bundles in the usual sense, these involutions coincide with those studied by Heller--Schaposnik \cite{heller&schaposnik}, Garcia-Prada--Wilkins \cite{garciaprada&wilkins} and Basu--Garcia-Prada \cite{basu&garciaprada}. This allow us to construct a degeneration of the subvarieties given by loci fixed by these involutions [Theorems \ref{tm degeneration of N} and \ref{tm degeneration of P}].

In the case of a K3 surface equipped with an antisymplectic involution the previous results provide degenerations of Prymian integrable system and their associated natural Lagrangian subarieties of the moduli of pure dimension $1$ sheaves on a K3. In particular, this gives a degeneration of the Prymian integrable systems constructed in \cite{ASF} by Arbarello--Sacc\`a--Ferretti [Section \ref{sc ASF}], and also [Section \ref{sc MT}] by Markushevich--Tikhomirov \cite{markushevich&tikhomirov}, Matteini \cite{matteini} and Sawon--Chen \cite{sawon&shen2, shen}. Our work then provides an answer to the question posed by Sawon in \cite{sawon} (which, strictly speaking, refers only to the case of primitive first Chern classes) and to the reformulation of Sawon's question on the non-primitve case. The later is, perhaps, a more interesting context than the first as Prymian integrable systems associated to primitive first Chern classes degenerate into subvarieties of rank $1$ Higgs moduli spaces, while those with non-primitive first Chern class correspond to moduli space of Higgs bundles of higher rank. 

We also review the degeneration given by Sawon--Shen \cite{sawon&shen} into the $\Sp(2m,\CC)$-Higgs moduli space, studying as well the degeneration of the associated natural Lagrangian [Section \ref{sc sawon and shen}]. It is worth noticing that we find that this natural Lagrangian degenerates into the fixed locus of an involution associated to $\U(m,m)$-Higgs bundles, whose Nadler--Langlads group is $\Sp(2m,\CC)$.

We also show that the Donagi--Ein--Lazarsfeld degenerations in the case of symplectic surfaces is equipped with a deformation of the symplectic structure of the moduli spaces [Theorem \ref{tm relative symplectic form}]. This allow us to understand the degenerations of  natural Lagrangian subvarieties studied in \cite{GYM} and the Prymian integrable systems as degenerations of $\BAA$ and $\BBB$-branes [Corollary \ref{co degeneration of branes}]. Finally, we prove that these branes are dual under a Fourier--Mukai transform restricted to the locus of pure $1$ dimension sheaves with smooth support curves [Proposition \ref{pr FM duals} and Corollary \ref{co FM duals}].

\subsection{Outline of the paper}

The paper is structured as follows. In Section \ref{sc preliminaries} we review the necessary background for our work. Section \ref{sc moduli of sheaves} contains generalities on the moduli space of pure dimension $1$ sheaves on surfaces. In Section \ref{sc natural involutions} we study the involutions on the moduli space naturally induced by pull-back of involutions on the surface. Section \ref{sc Prymian fibrations} is dedicated to the study of the involution on the moduli space of pure dimension $1$ sheaves induced by dualizing the restriction of a sheaf to its fitting support. We also study the composition of this involution with a natural involution, and study their fixed loci, which is generically a fibration by Prym varieties. In Section \ref{sc Prymian integrable systems} we consider these involutions starting from an anti-symplectic involution on a K3 surface, revisiting the Markushevich--Tikhomirov construction of Prymian integrable systems. Since we are interested on moduli spaces of sheaves with a non-primitive first Chern class, we treat this case in detail. Also, we study the Lagrangian subvarieties that arise from the fixed point locus of the natural involutions on the moduli constructed out of our anti-symplectic involution on the K3 surface. In Section \ref{sc ruled surfaces} we collect the necessary facts about ruled surfaces which will be used in Section \ref{sc L-Higgs bundles} to study twisted Higgs bundles, whose spectral data provide a particularly relevant example of pure dimension $1$ sheaves.

We describe in Section \ref{sc involutions and Higgs bundles} certain involutions on moduli spaces of twisted Higgs bundles. In Section \ref{sc involutions on ruled surfaces} we describe some involutions on a ruled surface and study their relation with the Poisson structure. We construct the corresponding involutions on the moduli spaces of spectral data of twisted Higgs bundles in Section \ref{sc natural involution under spectral correspondence} and study their behaviour under the spectral correspondence. We obtain involutions on the moduli space of $L$--Higgs bundles, which in the particular case of the twisting by the canonical bundle, are involutions that have been widely studied by Heller--Schaposnik, Garcia-Prada--Wilkins and Basu--Garcia-Prada.

The main results of the paper are contained in Section \ref{sc degenerations}. In Section \ref{sc degeneration of DEL} we provide a generalization to the case of an arbitrary smooth projective surface of the Donagi--Ein--Lazarsfeld degeneration, originally described for K3 surfaces. Hence, we obtain a degeneration of the moduli space of pure dimension $1$ sheaves on a surface into the moduli spaces of Higgs bundles twisted by the normal bundle of a curve inside our surface. In Section \ref{sc symplectic structure} we provide a deformation of the symplectic structure of the moduli spaces involved in the Donagi--Ein--Lazarsfeld degeneration in the case of symplectic surfaces, showing that it provides a degeneration of symplectic varieties. In Section \ref{sc degenerating involutions}, we study the behaviour under the Donagi--Ein--Lazarsfeld degeneration of the involutions considered in Sections \ref{sc natural involutions} and \ref{sc Prymian fibrations}. We then obtain a degeneration of the subvarieties described by their fixed loci. In Section \ref{sc degeneration of ASF systems} we provide an explicit description of such degenerations in the context of a K3 surface equipped with an anti-symplectic involution. We describe the non-linear degenerations of the Prymian integrable systems constructed by Arbarello--Sacc\`a--Ferretti, Markushevich--Tikho\-mi\-rov, Matteini amd Sawon--Shen, showing that they degenerate into integrable systems related to the Hitchin system. We also describe the degeneration of the associated natural Lagrangian subvarieties.

Finally, in Section \ref{sc branes}, we consider the natural Lagrangian subvarieties and the Prymian integrable systems in the context of branes, and provide some evidence for their duality.

\subsubsection*{Acknowledgments}

The author is thankful to Marcos Jardim for his encouragement to address this project, and specially to Gregoire Menet and Justin Sawon, for their comments, questions and suggestions. The author wants to thank Justin Sawon also for having shared a draft of the article \cite{sawon&shen}, which was useful for detecting some previous errors.

\section{Pure dimension 1 sheaves on surfaces and involutions on their moduli}

\label{sc preliminaries}

\subsection{Generalities on the moduli space of sheaves on surfaces}

\label{sc moduli of sheaves}

Let $S$ be a projective surface and take $\H$ to be a polarization on it. We say that a coherent sheaf $\Ff$ on $S$ is pure dimension $d$ if its schematic support has dimension $d$ and every subsheaf is also supported on dimension $d$ subschemes. In that case, its Hilbert polynomial $P(\Ff,\H)$ has degree $d$ and we define its $\H$-polarized rank $\rk(\Ff, \H)$ to be the leading term of $P(\Ff,\H)$ multiplied by $d!$. A coherent sheaf $\Ff$ is {\it $\H$-stable} ({\it resp.} {\it $\H$-semistable}) if it is of pure dimension, and for every proper subsheaf $\Ff' \subset \Ff$ we have that its Hilbert polynomial satisfy $P(\Ff',\H)/\rk(\Ff',\H) < P(\Ff,\H)/\rk(\Ff,\H)$ ({\it resp.} $P(\Ff',\H)/\rk(\Ff',\H) \leq P(\Ff,\H)/\rk(\Ff,\H)$) when $n \gg 0$. A semistable sheaf is polystable if it decomposes as a direct sum of stable sheaves. Simpson provided in \cite{simpson1} the existence of the moduli space $\M_{S}^\H(P)$ of pure dimension $\H$-semistable sheaves with Hilbert polynomial $P$, whose closed points represent polystable sheaves. 

The topological invariants of a sheaf over s smooth surface $S$ determined by the Hilbert polynomial are the $\H$-polarized rank $\rk(\Ff,H)$, the first Chern class $c_1(\Ff)$, and the Euler characteristic $\chi(\Ff)$. To simplify the computations, we introduce the {\it Mukai vector}, $\v(\Ff) := \ch(\Ff) \cdot \sqrt{\Todd(S)}$. Note that $\v(\Ff) \in H^{2*}(S, \ZZ)$  and the cup product in cohomology provides the {\it Mukai pairing} which endows $H^{2*}(S, \ZZ)$ with a lattice structure. We abbreviate by $\M^{\H}_S(\v_a)$ the moduli space of $\H$-semistable sheaves $S$ with Hilbert polynomial determined by the Mukai vector $\v \in H^{2*}(S, \ZZ)$.

It follows from deformation theory that the smooth locus of $\M_{S,\H}(\v_a)$ is the locus defined by those sheaves that are simple and that the tangent space to $\M_{S,\H}(\v_a)$ at the point determined by the simple coherent sheaf $\Ff$ corresponds with
\[
\Tt_\Ff \M_{S,\H}(\v_a) = \Ext^1_{S}(\Ff, \Ff),
\]
and its cotangent space is
\[
\Tt^*_\Ff \M_{S,\H}(\v_a) = \Ext^1_{S}(\Ff, \Ff \otimes K_S).  
\]
Another consequence of deformation theory is that $\M_{S,\H}(\v_a)$ is smooth at those points corresponding with simple sheaves. 

If, further, $S$ is a K3 or abelian surface, its canonical bundle $K_S$ is trivial and a choice of a (non-vanishing) section $\Omega_S$ provides a symplectic form on $S$. In both cases we say that $S$ is a {\it symplectic surface}. Thanks to Serre duality, $H^2(S, \Oo_S)$ is dual to $H^0(S, K_S) \cong \CC$.

\begin{theorem}[\cite{mukai1}]
Let $S$ be a symplectic surface. Then, $\M_{S,\H}(\v_a)$ is equipped with a holomorphic $2$-form $\Omega_\M$ defined by taking the trace of the Yoneda product composed with $\Omega_S$,
\[
\Ext^1_{S}(\Ff, \Ff) \wedge \Ext^1_{S}(\Ff, \Ff) \stackrel{\circ}{\longrightarrow} \Ext^2_{S}(\Ff, \Ff) \stackrel{\tr}{\longrightarrow}H^{2}(S,\Oo_{S}) \stackrel{\Omega_S(\cdot)}{\longrightarrow} \CC.
\]
Furthermore, $\Omega_\M$ is non-degenrate on the smooth locus of $\M_{S,\H}(\v_a)$, defining a symplectic form there.
\end{theorem}

Bottaccin \cite{bottacin_1} and Markman \cite{markman} generalized this construction to the case of Poisson surfaces, {\it i.e.} those equipped with Poisson bi-vector, a non-zero section $\Theta_S \in H^0(S, K_S^{-1})$.

\begin{theorem}[\cite{bottacin_1} and \cite{markman}] \label{tm Bottaccin-Markman Poisson str}
Let $S$ be is equipped with a Poisson bi-vector $\Theta_S$. Then, one can define a (closed although possibly degenerate) Poisson structure $\Theta_\M$ on the smooth locus of $\M_{S,\H}(\v_a)$ by taking the trace of the Yoneda product composed with $\Theta_S$,
\[
\Ext^1_{S}(\Ff, \Ff\otimes K_S) \wedge \Ext^1_{S}(\Ff, \Ff \otimes K_S) \stackrel{\circ}{\longrightarrow} \Ext^2_{S}(\Ff, \Ff\otimes K_S^2) \stackrel{\tr}{\longrightarrow}H^{2}(S,K_S^2) \stackrel{\langle \cdot , \Theta_S \rangle}{\longrightarrow} H^{2}(S,K_S) \cong \CC.  
\]
\end{theorem}

Note that a symplectic surface is also Poisson, and, in that case, the Bottacin--Markman form coincides with the Mukai form after identifying the tangent and cotangent spaces on the smooth locus of our moduli (via a trivialization of $K_S$).

\begin{proposition}[\cite{bottacin_1}] 
A smooth Poisson surface is either a symplectic surface ({\it i.e.} K3 or abelian) or a ruled surface. 
\end{proposition}

On a smooth surface $S$, every pure dimension $1$ sheaf $\Ff$ has a locally free resolution of length $2$ and we define the {\it fitting support}, $\supp(\Ff)$, of $\Ff$ as the determinant associated to this resolution. Note that the first Chern class is the cohomology class of its fitting support 
\[
c_1(\Ff) = \left [ \supp(\Ff) \right ]. 
\]
The Mukai vector of a pure dimension $1$ sheaf $\Ff$ on a smooth surface with canonical sheaf $K_S$ takes the form
\[
\v(\Ff) = \left ( 0, \left [ \supp(\Ff) \right ], \chi(\Ff) + \frac{1}{2} \left [ K_S \right ] \cdot \left [ \supp(\Ff) \right ]  \right ).
\]
and the pairing is simply given by the intersection of the supports,
\[
\langle \v(\Ff) , \v(\Ff') \rangle = \v_2(\Ff) \cdot \v_2(\Ff') = \left [ \supp(\Ff) \right ] \cdot \left [ \supp(\Ff') \right ].
\]
Pick a smooth curve $C$ of genus $g_C$ in $S$ and consider the associated curve $nC$ which is non-reduced if $n>1$. For any integer $a$ we define the Mukai vector 
\begin{equation} \label{eq definition of v_a}
\v_a = \left ( 0 , [nC], a - n^2 (C \cdot C) \right ).
\end{equation}
After Le Poitier \cite{lepotier}, the corresponding moduli space can be equipped with a fibration to the Hilbert scheme classifying dimension $1$ subschemes of $S$ with first Chern class equal to the second component of the Mukai vector, 
\begin{equation} \label{eq LePoitier}
\morph{\M_{S,\H}(\v_a)}{\Hilb_S([nC])}{\Ff}{\supp(\Ff).}{}{\h_S}
\end{equation}
The fibre of $\h_S$ over the curve $A \in \Hilb_S([nC])$ is the Simpson compactified Jacobian classifying $\H|_A$-semistable pure dimension $1$ sheaves on $A$ of rank $1$ and degree $a$
\[
\h_S^{-1}(A) = \overline{\Jac}^{\, \H}_{A}(a).
\]
%

\begin{remark} \label{rm Beauville-Mukai}
The Picard group of a smooth projective K3 surface is discrete and embeds into its second integral cohomology space. Hence, in this case, the second component of the Mukai vector fixes the determinant and $\Hilb_S([nC])$ is the linear system $|nC|$. 
After the work of Beauville \cite{beauville_fibr} the fibres are Lagrangian with respect to the Mukai form $\Omega_\M$. Hence, this is an algebraic completely integrable system called the {\it Beauville--Mukai system}.
\end{remark}

In general $\Pic(S)$ does not embed into $H^2(S,\ZZ)$, and $\Hilb_{S}([nC])$ can be reducible and even non-connected. Choose a curve $A$ parametrized by $\Hilb_{S}([nC])$, and let us denote by $\{ A \}$ the connected component of the Hilbert scheme containing the curve $A$. Accordingly, we denote by
\[
\M_{S,\H}(\v_a,A) = \h^{-1}_S \left ( \{ A \} \right ),
\]
the associated connected component of the moduli.

\begin{remark} 
Our notation differs from that of Matteini in \cite{matteini}, where, given a smooth curve $A$, he defined $\{ A \}$ to be the irreducible component of $\Hilb_{S}([nC])$ containing $A$. Since we will allow $A$ to be singular, and even non-reduced, $A$ may be contained in several irreducible components, although it determines uniquely its connected component. 
\end{remark}

\begin{remark} \label{rm vanishing irregularity}
When the irregularity of the surface $S$ vanishes, {\it i.e.} $h^1(S, \Oo_S) = 0$, we have that $\Hilb_{S}([nC])$ is a discrete union of connected components. In this case, $\{ A_i \} = |A_i|$, each being a linear system.
\end{remark}

Let us denote by $\Aa$ the restriction to $\{ A \} \times S$ of the universal subscheme associated to $\Hilb_S([nC])$. Note that $\Aa$ is a family of planar curves parametrized by $\{ A \}$. The polarization $\H$ on $S$ provides a relative polarization on $\Aa$ that we still denote by $\H$. By means of the Le Poitier fibration \eqref{eq LePoitier}, one can identify the component $\M_{S,\H}(\v_a,A)$ of our moduli space of pure dimension $1$ sheaves on $S$ with the relative compactified Jacobian over $\Aa$ of relative torsion free sheaves of rank $1$ and degree $a$,
\[
\M_{S,\H}(\v_a,A) \cong \overline{\Jac}^{\, \H}_{\Aa / \{ A \}}(a).
\]
If $\{ A \}^\sm$ denotes the open set of smooth curves, and $\Aa^\sm$ the restriction of $\Aa$ there, one can has the following identification of open subsets
\[
\left . \M_{S,\H}(\v_a,A) \right |_{\{ A \}^\sm} \cong \Jac_{\Aa^\sm / \{ A \}^\sm}(a),
\]
contained in the smooth locus. 

\begin{remark} \label{rm stability on the smooth locus}
We have dropped the polarization from the notation on the Jacobian as pure dimension $1$ sheaves whose restriction to its support have rank $1$, are always stable, regardless of the choice of polarization. 
\end{remark}



\subsection{Natural involutions from involutions on the surface}
\label{sc natural involutions}

In this section we study the involutions on the moduli space naturally induced by pull-back of involutions on the surface.

Suppose that our smooth projective surface $S$ is equipped with an involution 
\[
\zeta_S : S \to S.
\]

Let $C$ be a smooth projective curve in $S$ which is preserved under $\zeta_S$ and consider a Mukai vector $\v_a$ of the form specified in \eqref{eq definition of v_a} built out the $\zeta_S$-invariant curve $C$. Then, one trivially has that $\v_a$ is $\zeta_S$-invariant as well, so one can construct a birrational involution on the moduli space,
\[
\birrat{\M_{S,\H}(\v_a)}{\M_{S,\H}(\v_a)}{\Ff}{\zeta_S^*\Ff,}{}{\widehat{\zeta}_S}
\]
that we call {\it natural involution} associated to $\zeta_S$. We denote the closure of its fixed point locus by
\[
\N_{S,\H}(\v_a) := \overline{\Fix(\widehat{\zeta}_S)}.
\]
If, further, the polarization $\H$ is $\zeta_S$-invariant, $\widehat{\zeta}_S$ is a biregular morphism  and its fixed locus is already closed. The support morphism \eqref{eq LePoitier} restricts to the locus $\Hilb_S([nC])^{\widehat{\zeta}_C}$ of curves preserved by the involution, which may have several components.

Note that every curve $A'$ in $\Hilb_S([nC])^{\widehat{\zeta}_C}$ inherits an involution 
\[
\zeta_{A'} : A' \to A',
\]
inducing another birrational involution on the compactified Jacobian,
\begin{equation} \label{eq hatimath C'}
\birrat{\overline{\Jac}^{\, \H}_{A'}(a)}{\overline{\Jac}^{\, \H}_{A'}(a)}{\Ee}{\zeta_{A'}^*\Ee.}{}{\widehat{\zeta}_{A'}}
\end{equation}
As the fibres of \eqref{eq LePoitier} are identified with the corresponding compactified Jacobian, one has
\[
\h_S^{-1}(A') \cap \N_{S,\H}(\v_a) = \overline{\Fix(\widehat{\zeta}_{A'})}.
\]

It has been shown (see \cite[Lemma 9]{baraglia&schaposnik} or \cite{GYM} for instance) that the involution induced on the moduli space has the same behaviour under the symplectic form as the starting one. The generalization to the case of Poisson surfaces is straight-forward. 

\begin{proposition} \label{pr natural involutions preserve behaviour under Poisson structure}
Consider a Poisson surface $S$ and suppose it is equipped with a Poisson involution $\zeta_S^+$ ({\it resp.} an anti-Poisson involution $\zeta_S^-$). Then, the natural (birrational) involution $\widehat{\zeta}_S^+$ is also Poisson ({\it resp.} $\widehat{\zeta}_S^-$ is anti-Poisson).
\end{proposition}

\begin{proof}
We have to prove that whenever our involution $\zeta_S^\pm$ preserves the (powers of the) canonical bundle $K_S$ and commutes or anti-commutes with $\Theta_S$, so does $\widehat{\zeta}_S^\pm$ with respect to $\Theta_\M$ over the open subset $U \subset \M_S^H(\v_a)$ of smooth points where $\widehat{\zeta}_S^\pm$ is biregular. 

The proof is very similar to that of \cite[Theorem 3.4]{GYM}. Suppose that $\Ff$ represents a point in $U$. Then, $(\zeta_S^\pm)^*\Ff$ lies in $U$ too. For each $i \geq 0$, consider the induced morphism in cohomology $R^i(\zeta_S^\pm)^*$, and note that it commutes with the Yoneda product. Also, one has that $R^i(\zeta_S^\pm)^*$ commutes with the trace morphism and Serre duality, appearing in the definition of the Bottacin--Markman form. The Poisson involution $\zeta_S^+$ commutes with the remaining morphism in that definition, $\langle \cdot , \Theta_S \rangle$, while the anti-Poisson involution $\zeta_S^-$ anticommutes with it. Then, we can see that $\zeta_S^+$ commutes with the composition of all these morphisms, which define the Poisson form $\Theta_\M$ while $\zeta_S^-$ anticommutes with it. Hence, the proof is complete.
\end{proof}

Suppose now that the quotient surface $T := \quotient{S}{\zeta_S}$ is smooth and denote the quotient map by 
\[
\q_S : S \to T.
\]
Suppose as well that  
\[
C := D \times_T S  
\]
for some smooth curve $D \subset T$. We have already seen that the Mukai vector $\v_a$ is $\zeta_S$-invariant as so is $C$. 

\begin{remark} \label{rm v_a eta-invariant}
A Mukai vector of the form $\v_{2b}$ is the pull-back of the Mukai vector in $T$,
\[
\w_b := \left ( 0, [nD], b - n^2 (D \cdot D) \right ).
\]
\end{remark}

Pick a curve $B \subset T$ in the class $[nD]$ and consider its lift $A = B \times_T S$ to $S$, lying in the class $[nC]$. Recall that we denote by $\{ B \}$ and $\{ A \}$ the connected components of $\Hilb_T([nD])$ and $\Hilb_S([nC])$ containing $B$ and $A$, respectively. Note that $\widehat{\zeta}_S$ restricts to $\{ A \}$ although the fixed lous $\{ A \}^{\widehat{\zeta}_S}$ might be disconnected. We shall denote by $\q_S^*\{B \}$ the connected subset of $\{ A \}^{\widehat{\zeta}_S}$ given by those curves obtained by lifting curves in $\{B \}$. In principle, $\q_S^*\{B \}$ is not a connected component of $\{ A \}^{\widehat{\zeta}_S}$, only a union of some irreducible components.

\begin{remark}
We have seen Remark \ref{rm vanishing irregularity} whenever the surface $S$ has vanishing irregularity, the $\{ A_i \}$ are the linear systems $|A_i|$. In this case, the locus $| A_i |^{\widetilde{\zeta}_S}$ of curves preseved by $\zeta_S$ is disconnected and amounts to the union of $|A_i|^{\widetilde{\zeta}_S}_+$ and $|A_i|^{\widetilde{\zeta}_S}_-$, each being the projectivization of the $+1$ and $-1$ eigenspaces for the action of $\widehat{\zeta}_S$ on $H^0(S, \Oo_S(A_i))$. In this context, 
\[
|A|^{\widetilde{\zeta}_S}_+ = \q_S^*|B|. 
\]
\end{remark}

We consider the subvariety
\begin{equation} \label{eq N over B}
\N_{S,\H}(\v,B) := \N_{S,\H}(\v_a) \cap \h^{-1}_S \left ( \q_S^*\{ B \} \right ).
\end{equation}
By restriction of \eqref{eq LePoitier}, $\N_{S,\H}(\v,B)$ is equipped with a fibration,
\begin{equation} \label{eq support morphism for N}
\xymatrix{
\N_{S,\H}(\v,B) \, \ar[d]_{\nu} \ar@{^(->}[rr] & & \M_{S,\H}(\v, A) \ar[d]^{\h_S}  \cong \overline{\Jac}^{\, \H}_{\Aa / \{ A \}}(a)
\\
\q_S^*\{ B \}\,  \ar@{^(->}[rr] & & \{ A \},
}
\end{equation}
whose fibres are
\[
\nu^{-1} (A') = \overline{\Fix(\widehat{\zeta}_{A'})}.
\]

Consider over $T$ the moduli space $\M_{T,\I}(\w_b)$ of sheaves with a Mukai vector $\w_b$ and a polarization $I$. As before, we denote by $\{ B \}$ the connected component of the Hilbert scheme in $T$ containing $B$, and by $\Bb$ the family of curves parametrized by $\{ B \}$ that we obtain from restricting there the universal subscheme associated to $\Hilb([nD])$. Recall that one has the identification 
\[
\M_{T,\I}(\w_b, B) \cong \overline{\Jac}^{\, \I}_{\Bb/\{ B \}}(b),
\]
Denote the pull-back of the polarization by
\[
\widetilde{\I} = \q_S^*\I.
\]
For these choices, the involution $\widehat{\zeta}_S$ is biregular and the image of the pull-back morphism under the corresponding quotient map lies in its fixed locus,
\begin{equation} \label{eq definition of hatq}
\morph{\M_{T,\I}(\w_b, B) \cong \overline{\Jac}^{\, \I}_{\Bb/\{ B \}}(b)}{\N_{S, \widetilde{\I}}(\v_{2b}, B) \subset \M_{S, \widetilde{\I}}(\v_{2b}, B) \cong \overline{\Jac}^{\, \widetilde{\I}}_{\Aa/\{ A \}}(2b)}{\Ff}{\q_S^*\Ff.}{}{\widehat{\q}_S}
\end{equation}
Given the curves $B \subset T$ and $A \subset S$ as before, denote by 
\[
\q_A : A \to B,
\]
the projection induced by $\q_S$. Observe that the restriction of \eqref{eq definition of hatq} to the Hitchin fibres gives,
\[
\widehat{\q}_A : \h_T^{-1}(B) \cong \overline{\Jac}^{\, \I}_B(b) \to \nu^{-1}(A) \subset \h_S^{-1}(A) \cong \overline{\Jac}^{\, \widetilde{\I}}_A(2b),
\]
where $\widehat{\q}_A := \q_A^*$.

Let us denote by $\{ B \}^\ssl$ the locus of $\{ B \}$ parametrizing smooth curves $B'$ whose lift $A' = B' \times_T S \subset \{ A \}$ is also smooth. On the contrary, in particular when $\zeta_S$ is without fixed points, $\q_{\Aa^\ssl}$ is unramified. The notation stands for {\it smooth and smooth lift}. Recall the family of curves $\Bb \to \{ B \}$ and denote by $\Bb^\ssl$ its restriction to $\{ B \}^\ssl$. Set also $\Aa^\ssl$ to be the restriction to $\q_S^*\{ B \}^{\ssl}$ of the family of curves $\Aa \to \{ A \}$, and observe that it comes equipped with a projection induced by $\q_S$,
\begin{equation} \label{eq definition of q_Aa}
\xymatrix{
\Aa^\ssl \ar[rr]^{\q_{\Aa^\ssl}}\ar[rd] & & \Bb^\ssl \ar[ld]
\\
& \{ B \}^\ssl. &
}
\end{equation}
One can see that $\q_{\Aa^\ssl}$ is ramified whenever the intersection of the branching locus $\Delta$ of $\q_S$ with generic elements of $\{ B \}$ is non-empty. Let us denote by $\widehat{\q}_{\Aa^\ssl}$ the restriction of $\widehat{\q}_S$ to $\q_S^*\{ B \}^\ssl$. Recall from Remark \ref{rm stability on the smooth locus}, that the sheaves supported on smooth (hence irreducible) curves are always stable,
\begin{equation} \label{eq definition hatq_Aa}
\widehat{\q}_{\Aa^\ssl} : \Jac_{\Bb^\ssl/\{ B \}^\ssl}(b) \longrightarrow \N_{S,\H}(\v_{2b}, B) \subset \Jac_{\Aa^\ssl/\q_S^*\{ B \}^\ssl}(2b).
\end{equation}
This provides a description of an open subset of $\N_{S,\H}(\v_{2b}, B)$.

\begin{proposition} \label{pr q embedding on ssl}
Suppose that $\q_{\Aa^\ssl}$ is ramified. Then 
\begin{equation} \label{eq q embedding on ssl}
\left . \N_{S,\H}(\v_{2b}, B) \right |_{\{ B \}^\ssl} \cong \Jac_{\Bb^{\, \ssl} / \{ B \}^\ssl}(b) 
\end{equation}
and \eqref{eq definition of hatq} is, generically, a 1:1 morphism.

If, on the contrary, $\q_{\Aa^\ssl}$ is an unramified $2:1$ cover associated to the relative $2$-torsion line bundle $\Ll \to \Bb^\ssl$, one has that $\widehat{\q}_{\Aa^\ssl}$ factors through the quotient by $\ZZ_2$ acting by tensor product with $\Ll$, giving
\begin{equation} \label{eq q factors through Z_2 on ssl}
\left . \N_{S,\H}(\v_{2b}, B) \right |_{\{ B \}^\ssl} \cong \quotient{\Jac_{\Bb^{\, \ssl} / \{ B \}^\ssl}(b)}{\ZZ_2} 
\end{equation}
and \eqref{eq definition of hatq} is, generically, a 2:1 morphism.
\end{proposition}

\begin{proof}
In the first case, by Kempf descent lemma, $\image \left (\widehat{\q}_{\Aa^\ssl} \right )$ coincides with the left-hand side of \eqref{eq q embedding on ssl}. Since, by hypothesis, $\q_{\Aa^\ssl}$ is ramified, $\widehat{\q}_{\Aa^\ssl}$ is an embedding. As we find ourselves within the smooth locus of $\M^{\H}_S(\v_{2b},A)$, $\widehat{\q}_{\Aa^\ssl}$ provides an isomorphism of its source with its image. As $\{ B \}^\ssl$ is a dense open subset of $\{ B \}$, this is, generically, $1:1$ morphism.

The second case follows since the pull-back of $\Ll$ is trivial over $\Aa^\ssl$.
\end{proof}

\subsection{The dualizing involution and Prymian fibrations}
\label{sc Prymian fibrations}

In this section we study the involution on the moduli space of pure dimension $1$ sheaves induced by dualizing the restriction of a sheaf to its fitting support. We also study the composition of this involution with a natural involution, and study their fixed loci, which is generically a fibration by Prym varieties.

Given a Mukai vector $\v_a$ on $S$ as defined in \eqref{eq definition of v_a}, whose second component is the class $[nC]$, let us choose a line bundle $J \in \Pic(S)$ satisfying whose intersection with the second component of the Mukai satisfies
\begin{equation} \label{eq condition for J and v}
\frac{2a}{n} =  J \cdot C + n (C \cdot C).
\end{equation}
Note that not every value of $a$ and $n$ allow for such a $J$. It is discussed in \cite[Sections 3.3 and 3.5]{ASF}, that the moduli space $\M_{S,\H}(\v_a)$ is equipped with an involution
\begin{equation} \label{eq eta}
\morph{\M_{S,\H}(\v_a)}{\M_{S,\H}(\v_a)}{\Ff}{\Ee \! xt^1_S(\Ff, J).}{}{\xi_{J,S}}
\end{equation}
Pick $A = \supp(\Ff)$, which is Gorenstein with canonical line bundle $K_S(A)|_{A}$, Grothendieck--Verdier duality implies that, 
\begin{equation} \label{eq fibrewise description of xi}
\Ee \! xt^1_S(\Ff, J) \cong \Hh \! om_{A}(\Ff|_{A}, J(A)|_{A}).
\end{equation}
Since $A$ belongs to the class $[nC]$, equation \eqref{eq condition for J and v} implies $2a = \deg(J(A)|_{A})$ and this ensures that $\xi_{J,S}$ preserves the Mukai vector. Observe that $\xi_{J,S}^2 = \id_M$ holds as every pure dimension $1$ sheaf in a smooth surface is reflexive. Also, $\xi_{J,S}$ restricts to the fibres of the Mukai fibration, $\h^{-1}_S(A) \cong \overline{\Jac}^{\, a}_{\H}(A)$, giving the {\it dualizing} involution composed with an appropriate tensor product,
\begin{equation} \label{eq eta_J C'}
\morph{\overline{\Jac}^{\, \H}_A(a)}{\overline{\Jac}^{\, \H}_A(a)}{\Ee}{\Ee^\vee \otimes J(A)|_{A},}{}{\xi_{J,A}}
\end{equation}
where we recall that the dual of a rank $1$ torsion-free sheaf on a Gorenstein curve is always well defined and $\H$-stability is preserved.

\begin{remark} \label{rm choice of J_0 and v_0}
Condition \eqref{eq condition for J and v} is always satisfied when we choose $\v_0$ and $J_0 = \Oo_S(-nC)$, for instance. In that case $\xi_{J_0, A}$ is just the dualizing involution $\Ee \mapsto \Ee^\vee$. 
\end{remark}

Recall that we have chosen the Mukai vector $\v_a$ to be $\zeta_S$-invariant. If further the line bundle $J$ on $S$ satisfying \eqref{eq condition for J and v} is $\zeta_S$-invariant too, $\widehat{\zeta}_S$ and $\xi_{J,S}$ commute and, following \cite{ASF}, we consider their composition 
\[
\lambda_{J,S} := \widehat{\zeta}_S \circ \xi_{J,S} : 
\M_{S,\H}(\v_a) \dashrightarrow \M_{S,\H}(\v_a),
\]
which defines a birrational involution on $\M_{S,\H}(\v_a)$. We denote the closure of its fixed locus by
\[
\P_{S,\H}(\v_a,J) := \overline{\Fix(\lambda_{J,S})}.
\]

We now study the relation of these involutions with the Poisson (and symplectic) structure. It was proved in \cite[Proposition 3.11]{ASF} that $\xi_{J,S}$ is anti-symplectic under the assumptions of $S$ being a K3 surface, $\v_a$ a Mukai vector with $n=1$ and $J = \Oo_S(-C)$. 

\begin{proposition} 
\label{pr eta antisymplectic}
Given a Poisson surface $S$, for any choice of Mukai vector $\v_a$ and $J \in \Pic(S)$ satisfying \eqref{eq condition for J and v}, the involution $\xi_{J,S}$ is anti-Poisson with respect to the Bottaccin--Markman Poisson form $\Theta_\M$ on $\M_{S,\H}(\v_a)$,
\[
\xi_{J,S}^* \, \Theta_\M = -\Theta_\M.
\]
\end{proposition}

\begin{proof}
The proof of \cite[Proposition 3.11]{ASF} extends to the Poisson case with minor changes.
\end{proof}

As an immediate consequence of Propositions \ref{pr natural involutions preserve behaviour under Poisson structure} and \ref{pr eta antisymplectic} one derives the following corollary which will be used in Section \ref{sc degeneration of ASF systems}.

\begin{corollary} \label{co lambda is Poisson}
Let $S$ be a Poisson surface and $\zeta_S^+$ a Poisson involution on it ({\it resp.} $\zeta_S^-$ an anti-Poisson involution). For any choice of Mukai vector $\v_a$ and a $\zeta_S^+$-invariant ({\it resp.} $\zeta_S^-$-invariant) $J \in \Pic(S)$ satisfying \eqref{eq condition for J and v}, the involution $\lambda_{J,S}^+ = \widehat{\zeta}_S^+ \circ \xi_{J,S}$ is a birrational anti-Poisson involution ({\it resp.} $\lambda_{J,S}^- = \widehat{\zeta}_S^- \circ \xi_{J,S}$ is a birrational Poisson involution) on $\M_{S,\H}(\v_a)$. 
\end{corollary}

Let us now assume that the quotient surface $T$ is smooth, and recall the notation introduced in Section \ref{sc natural involutions}. Consider the closed subvariety
\begin{equation} \label{eq P over B}
\P_{S,\H}(\v_a,J, B) := \P_{S,\H}(\v_a,J) \cap \h^{-1}_S \left ( \q_S^*\{ B \} \right ).
\end{equation}
Whenever $\H$ is $\zeta_S$-invariant, $\lambda_{J,S}$ is well defined as a regular involution and its fixed locus is already closed.

By construction, one has the following commuting diagram
\begin{equation} \label{eq support morphism for P}
\xymatrix{
\P_{S,\H}(\v_a,J,B) \, \ar[d]_{\mu} \ar@{^(->}[rr] & & \M_{S,\H}(\v_a, A) \ar[d]^{\h_S} \cong \overline{\Jac}^{\, \H}_{\Aa / \{ A \}}(a)
\\
\q_S^*\{ B \}\,  \ar@{^(->}[rr] & &  \{ A \}. 
}
\end{equation}
Observe that, over $A' \in \q_S^*\{ B \}$, 
\[
\mu^{-1} (A') = \overline{\Fix(\lambda_{J,A'})},
\]
where 
\[
\birrat{\overline{\Jac}^{\, \H}_{A'}(a)}{\overline{\Jac}^{\, \H}_{A'}(a)}{\Ee}{\zeta_{A'}^* \Ee^\vee \otimes J(A')|_{A'},}{}{\lambda_{J,A'} := \widehat{\zeta}_{A'} \circ \xi_{J,A'} }
\]
is the induced birrational involution on $\h^{-1}_S(A')$. 
If $A'$ is irreducible, in particular when it is smooth and connected, every pure dimension sheaf is stable so $\widehat{\zeta}_{J,A'}$ is a biregular involution and its fixed locus is already closed.  

Recall that we denoted by $\{ B \}^{\ssl}$, the open subset of $\{ B \}$ parametrizing smooth curves, whose lift to $S$ is smooth too.

\begin{remark} \label{rm smooth fibres are Prym varieties}
For $B' \in \{ B \}^\ssl$ lifting to $A'$ smooth one has that $\h^{-1}_S(A') = \Jac_{A'}(a)$ is a torsor for the abelian variety $\Jac_{A'}(0)$. After \eqref{eq eta_J C'}, one has that the fibre of $\mu$ is 
\[
\mu^{-1}(A') = \ker \left (J(A')|_{A'} + \widehat{\zeta}_{A'} \right ),
\]
whose connected components are torsors for the Prym variety 
\[
\Prym(\q_{A'}) = \ker \left (\id_{\Jac} + \widehat{\zeta}_{A'} \right )_0.
\]
\end{remark}

Recall from \eqref{eq definition of q_Aa} the morphism $\q_{\Aa^\ssl}$. If $\q_S$ is ramified and its branching locus $\Delta$ intersects $B$ non-trivially, then $\q_{\Aa^\ssl}$ is a ramified morphism. Let consider the case described in Remark \ref{rm choice of J_0 and v_0}.

\begin{remark} \label{rm Prym at ssl}
Consider the Mukai vector $\v_0$ associated to $a = 0$ and $J_0 = \Oo_S(-nC)$. Recall that for ramified maps, the kernel $\id_{\Jac} + \widehat{\zeta}_{A'}$ is connected. It follows immediately from Remark \ref{rm smooth fibres are Prym varieties} and the fact that $\q_{\Aa^\ssl}$ is ramified, that
\begin{equation} \label{eq P over ssl}
\left . \P_{S,\H}(\v_0, J_0,B) \right |_{\{ B \}^{\ssl}} \cong \Prym \left (\q_{\Aa^\ssl}/{\{ B \}^{\ssl}}\right ).
\end{equation}
Since the kernel $\id_{\Jac} + \widehat{\zeta}_{A'}$ has two connected components if $\q_{A'}$ is an unramified $2:1$ cover. Then, in that case, $\P_{S,\H}(\v_0, J_0,B) |_{\{ B \}^{\ssl}}$ has $2$ connected components and \eqref{eq P over ssl} holds after restricting ourselves to the connected component of the identity.
\end{remark}

Whenever $\{ B \}^\ssl$ is a dense open subset of $\{ B \}$, by Remark \ref{rm Prym at ssl} the generic fibres of $\mu : \P_{S,\H}(\v_a,J,B) \to \q_S^*\{ B \}$ are (non-canonically) isomorphic to Prym varieties. In that case, we refer to it as the {\it Prymian fibration} associated to the involution $\zeta_S$.


\subsection{Natural Lagrangian and Prymian integrable systems}

\label{sc Prymian integrable systems}

Following Sawon--Shen \cite{sawon&shen} (see also \cite{shen}), we provide a straight-forward adaptation of the Markushevich--Tikhomirov construction of Prymian integrable systems to the case of non-primitive first Chern classes.

Prymian integrable systems are integrable systems fibered by Prym varieties. The class that we consider here arises from an anti-symplectic involution on a K3 surface. These systems were first constructed by Markushevich and Tikhomirov \cite{markushevich&tikhomirov} and later generalized by Arbarell, Sacc\`{a} and Ferretti \cite{ASF}, Matteini \cite{matteini}, Sawon and Shen \cite{sawon&shen2, shen} and others. The case considered in \cite{markushevich&tikhomirov} corresponds to an anti-symplectic involution on a smooth K3 surface, whose associated quotient is a del Pezzo surface of degree $2$. In \cite{ASF}, Arbarello, Sacc\`{a} and Ferretti constructed a Prymian integrable system from an anti-symplectic involution on a K3 giving an Enriques surface. Matteini \cite{matteini} generalized the construction of \cite{markushevich&tikhomirov} to the case of anti-symplectic involutions on K3 associated to a del Pezzo surface of degree $3$. Sawon and Shen \cite{sawon&shen2, shen} studied the case of an antisymplectic involution on a K3 whose quotient is a del Pezzo surface of degree $1$.


Let $X$ denote a smooth K3 surface, $\zeta_X^-$ an antisymplectic involution on it, and $Y$ the quotient $Y = X/\zeta_X^-$ which is a either a smooth rational surface or an Enriques surface. Over $Y$, pick a smooth connected curve $D \subset Y$ of genus $g_D \geq 2$ and consider
\begin{equation} \label{eq C from D}
C = D \times_Y X = \q_S^{-1}(D),
\end{equation}
which is smooth, and it is equipped with an double cover $\q_C: C \to D$ ramified at $K_Y \cdot D$ and with a Galois involution $\zeta_C : C \to C$ induced by $\zeta_X^-$. Similarly, consider curves $nD$ and $nC = nD \times_Y X$ which are non-reduced for $n>1$, and denote by $g_{nD}$ and $g_{nC}$ the genus of the smooth curves in the linear systems $|nD|$ and $|nC|$. We have for all $n > 0$, that 
\begin{equation} \label{eq relation of the genuses}
g_{nC} = 2g_{nD} -1 - nD \cdot K_Y,
\end{equation}
where, 
\[
g_{nD} = 1 + \frac{nD \cdot (nD + K_Y)}{2},
\]
by the genus formula. Since $\chi(Y) = 1$ in all possible cases, 
\[
\chi(nD) = 1 + \frac{nD \cdot (nD - K_Y)}{2},
\]
so
\[
\chi(nD) = g_{nD} - nD \cdot K_Y.
\]
Combining this with \eqref{eq relation of the genuses} one has
\begin{equation} \label{eq relation of the dimensions}
g_{nC} - g_{nD} = \chi(nD)-1.
\end{equation}
Since the irregularity of $Y$ vanishes, $h^1(Y,\Oo_Y) = 0$, $\Hilb_{Y}([nD])$ consists on a disjoint union of linear systems, for instance $\{ nD \} = |nD|$.



Consider the natural (birrational) anti-symplectic involution  $\widehat{\zeta}_X^{\, -}$ and, having picked $\J$ for which \eqref{eq condition for J and v} holds with respect to $\v_a$, the (birrational) involution $\lambda_{\J,X}^- = \widehat{\zeta}_X^{\, -} \circ \xi_{\J, X}$, which is symplectic. Pick the subvarieties $\N_{X,\H}^{\, -}(\v_a,nD)$ and $\P_{X,\H}^{\, -}(\v_a, \J, nD)$ defined as the closure of the fixed locus of $\widehat{\zeta}_X^{\, -}$ and $\lambda_{\J,X}^-$ respectively, given by those sheaves supported on curves parametrized by $\q_X^*|nD|$.

Recall that we denoted by $|nD|^\ssl$ the locus of smooth curves with smooth lift in $|nD|$. 

\begin{lemma}
Let $D$ ne a smooth curve in a smooth surface $Y$ and let $\q_X : X \to Y$ be a $2:1$ covering by the smooth surface $X$. Then, $|nD|^\ssl$ is an open subset in $|nD|$.
\end{lemma}

\begin{proof}
By Bertini's theorem, the locus of smooth curves, $|nD|^\sm$, is an open subset in $|nD|$. If the curve $B'$ is smooth and does not intersect tangentially the branch locus $\Delta$ of $\q_X$, the lifted curve $A' = B' \times_Y X$ is smooth by construction. The later is an open condition, so $|nD|^\ssl$ is an open subset of $|nD|$. 
\end{proof}

Hence $\q_X^*|nD|^\ssl$ is open in $\q_X^*|nD|$, and so are $\N_{X,\H}^{\, -}(\v_a,nD)|_{\q_S^*|nD|^\ssl}$ in $\N_{X,\H}^{\, -}(\v_a,nD)$ and $\P_{X,\H}^{\, -}(\v_a, \J, nD)|_{\q_S^*|nD|^\ssl}$ in $\P_{X,\H}^{\, -}(\v_a, \J, nD)$, respectively. In that case, $\h_X^{-1}(A)$ classifies line bundles over $A$, being stable and smooth as points in $\M_{X,\H}(\v_a)$. Then, both $\widehat{\zeta}_X^{\, -}$ and $\lambda_{\J,X}^-$ restrict to biregular involutions over the locus of smooth curves in $\q_X^*|nD|$. This allow us to prove the following.

\begin{proposition} \label{pr N is Lagrangian}
The projective subvariety $\N_{X,\H}^{\, -}(\v_a,nD) \subset \M_{X,\H}(\v_a)$ is Lagrangian with respect to $\Omega_\M$, in particular,
\[
\dim \left ( \N_{X,\H}^{\, -}(\v_a, nD) \right ) = n^2(g_C-1) + 1.
\]
\end{proposition}

\begin{proof}
We have seen that $\M_{X,\H}(\v_a)$ is smooth over $\q_X^*|nD|^\ssl$. Also, $\widehat{\zeta}_X^{\, -}$ is a biregular anti-symplectic involution there. It then follows that its fixed point locus $\N_{X,\H}^{\, -}(\v_a,nD)|_{\q_S^*|nD|^\ssl}$ is smooth and Lagrangian. As it is an open subset of $\N_{X,\H}^{\, -}(\v_a,nD)$, it follows that the later is a Lagrangian subvariety as well. 
\end{proof}

In view of Proposition \ref{pr N is Lagrangian}, we refer to $\N_{X,\H}^{\, -}(\v_a,nD)$ as the {\it natural Lagrangian} subvariety of $\M_{X,\H}(\v_a)$ associated to $\zeta_X^{\, -}$.

We study now the fixed locus of $\lambda_{\J,X}^-$.

\begin{proposition} \label{pr dimension of P}
The projective subvariety $\P_{X,\H}^{\, -}(\v_a, \J,nD) \subset \M_{X,\H}(\v_a)$ has dimension
\[
\dim \left ( \P_{X,\H}^{\, -}(\v_a, \J, nD) \right ) = \chi(nD) +  h^0(Y,nD) - 2 = 2 h^0(Y,nD) - 2 + (h^2(Y,nD) - h^1(Y,nD)).
\]
\end{proposition}

\begin{proof}
Following Remark \ref{rm Prym at ssl} one has a description of the open subset $\P_{X,\H}^{\, -}(\v_a, \J, nD)|_{\q_S^*|nD|^\ssl}$ of $\P_{X,\H}^{\, -}(\v_a, \J, nD)$ in terms of a Prymian fibration constructed with a ramified projection of curves with genus $g_{nC}$, to curves of genus $g_{nD}$, having fibres of dimension $g_{nC} - g_{nD}$. After \eqref{eq relation of the dimensions}, the statement follows. 
\end{proof}

\begin{remark} \label{rm Lagrangian condition}
Following Proposition \ref{pr dimension of P}, whenever 
\begin{equation} \label{eq Lagrangian condition}
h^1(Y,nD) = h^2(Y,nD),
\end{equation}
the dimension of $\P_{X,\H}^{\, -}(\v_a, \J, nD)$ is $2\chi(nD) - 2$ and the generic fibres of $\mu$ are half-dimensional.
\end{remark}

We next see that the Prymian fibration
\[
\mu : \P_{X,\H}^{\, -}(\v_a, \J, nD) \to \q_X^*|nD|
\]
endows $\P_{X,\H}^{\, -}(\v_a, \J, nD)$ with the structure of an integrable system, that we call the {\it Prymian integrable system} associated to the anti-symplectic involution $\zeta_X^-$. This construction has been considered by several authors \cite{markushevich&tikhomirov, ASF, matteini, sawon&shen2}, although in all these cases the first Chern class is chosen to be primitive. As we are interested in the non-primitive case, we include a proof of it instead of just citing the previous articles.

\begin{theorem} \label{tm Prymian integrable systems}
Let $X$ be a smooth K3 surface $X$ equipped with an anti-symplectic involution $\zeta_X^-$ giving the quotient $Y = X/\zeta_X^-$. Pick a smooth curve $D \subset Y$ of genus $g_D \geq 2$ and $n \geq 1$ such that \eqref{eq Lagrangian condition} holds. Take $C$ as in \eqref{eq C from D}, a Mukai vector $\v_a$ as in \eqref{eq definition of v_a} and $\J \in \Pic(X)$ satisfying \eqref{eq condition for J and v} with respect to $\v_a$. Then, one has that
\begin{enumerate}
    \item \label{it dimension} $\P_{X,\H}^{\, -}(\v_a,\J,nD)$ is a projective variety of dimension $2\chi(nD) - 2 = n^2D^2 -nD\cdot K_Y$;
    \item \label{it Lagrangian} the smooth locus of $\P_{X,\H}^{\, -}(\v_a,\J,nD)$ carries a holomorphic $2$-form $\Omega_\P$ which is symplectic on a dense open subset where the fibres of $\mu$ are Lagrangian; and
    \item \label{it Prymian} the generic fibre of $\mu$ is a a torsor over a smooth abelian variety of dimension $\chi(nD) - 1 = \frac{1}{2}(n^2D^2 -nD\cdot K_Y)$, being the Prym of a double cover of smooth curves.  
\end{enumerate}
\end{theorem}

\begin{proof}
\eqref{it dimension} is an immediate consequence of Proposition \ref{pr dimension of P} and Remark \ref{rm Lagrangian condition}.

Over the smooth locus of $\P_{X,\H}^{\, -}(\v_a, \J, nD)$, its tangent space embeds into the tangent space of $\M_{X,\H}(\v_a)$ and one can define the holomorphic $2$-form by restricting the Mukai form $\Omega_\M$ there. Also, $\lambda_{\J,X}^-$ restricts to a biregular symplectic involution over the locus of smooth curves in $\q_X^*|nD|$, where the symplectic form $\Omega_\M$ is well defined. There, the open subset of $\P_{X,\H}^{\, -}(\v_a, \J, nD)$ given by the fixed locus of $\lambda_{\J,X}^-$ inherits a symplectic form which obviously coincides with the restriction of $\Omega_\P$ whenever we find ourselves in the smooth locus of $\P_{X,\H}^{\, -}(\v_a, \J, nD)$. By \cite{beauville_fibr}, $\Omega_\M$ vanishes on the fibres of $\h_X$ as we have seen in Remark \ref{rm Beauville-Mukai}. As $\Omega_\P$ is obtained from restricting $\Omega_\M$, it follows that $\Omega_\P$ vanishes on the fibres of $\mu$ over smooth curves. Equivalently, for every $A \in \q_X^*|nD|$ smooth, $\mu^{-1}(A)$ is isotropic with respect to $\Omega_\P$. Recalling from Remark \ref{rm Lagrangian condition} that these fibres are half-dimensional, we conclude the proof of \eqref{it Lagrangian}. 

Finally \eqref{it Prymian} follows from Remark \ref{rm smooth fibres are Prym varieties}.
\end{proof}


\subsection{Ruled surfaces}

\label{sc ruled surfaces}

Let us recall in this section some facts about ruled surfaces that will be used in Section \ref{sc L-Higgs bundles}.

Given a smooth projective curve $C$ of genus $g_C$ and a line bundle $L$ on $C$, consider the total space of our line bundle $\Tot(L)$ and the obvious projection $p: \Tot(L) \to C$. Set $\ell := \deg(L)$, assuming $\ell \geq 0$, and consider the projective compactification of $\Tot(L)$, namely the ruled surface with topological invariant $\ell$, 
\begin{equation} \label{eq definition of LL}
\LL := \PP(\Oo_C \oplus L), 
\end{equation}
naturally equipped with the projection that we still denote by $p : \LL \longrightarrow C$. 


One has, of course, 
\[
\Tot(L) = \LL - \{ \sigma_\infty \},
\]
where $\sigma_\infty \cong C$ is the curve in $\LL$ associated to the $0$ section of $\Oo_C$. Let us briefly recall some properties of $\LL$ following \cite[Section 2, Chapter V]{hartshorne}. The curve $\sigma_\infty$ has negative self-intersection, hence  
\begin{equation} \label{eq multiples of sigma infty}
H^0(\LL, \Oo_\LL(m \sigma_\infty)) = 1
\end{equation}
for $m \geq 0$. Along with $\Pic(C)$, $\sigma_\infty$ generates the Picard group of $\LL$,  
\begin{equation} \label{eq Pic of KK}
\Pic(\LL) = \left ( \ZZ \cdot \sigma_{\infty} \right ) \oplus p^* \Pic(C),
\end{equation}
so any divisor $D$ on $\LL$ is numerically equivalent to $b_1 \sigma_\infty + b_2 F$, where $F$ denotes the fibre and $\sigma_\infty \cdot \sigma_{\infty} = -\ell$, $\sigma_\infty \cdot F = 1$ and $F \cdot F = 0$. A divisor is ample if and only if $b_1 > 0$ and $b_2 > \ell b_1$, and only exist irreducible curves in those classes with $b_1 > 0$ and $b_2 \geq \ell b_1$, with the exception of the infinity section $\sigma_\infty$. The zero section of $L$ defines $\sigma_0 : C \to \LL$ whose image is linearly equivalent to $\sigma_\infty + \ell F$ and its normal bundle returns
\begin{equation} \label{description of self intersection of sigma_0}
L \cong \Oo_{\LL}(\sigma_0)|_{\sigma_0}.
\end{equation}

Observe that any divisor $D$ with null intersection with $\sigma_\infty$ is numerically equivalent to a multiple of $\sigma_0$. Also, it is linearly equivalent to $n\sigma_0$ if and only if $\Oo_\LL(D)$ has a non-zero section, as by scaling the fibres of $\Tot(L)$, $D$ deforms linearly to a multiple curve supported on $\sigma_0$. Hence, the linear system $|n\sigma_0|$ has an open subset classified by linear deformations of the curve, which coincides with those curves not intersecting $\sigma_\infty$. Linear deformations of a curve are classified by the sections of the normal bundle restricted to the curve. Hence, one has the identification
\[
|n\sigma_0|_{\supp \cap \sigma_\infty = \emptyset} = H^0(n\sigma_0, \Oo_{\LL}(n\sigma_0)|_{n\sigma_0}).
\]
Let $r$ denote the obvious projection of the multiple curve onto its reduced support and observe that $r_*\Oo_{\LL}(n\sigma_0)|_{n\sigma_0}$ amounts to $\Oo_{\LL}(n\sigma_0)|_{\sigma_0} \otimes r_*\Oo_{n\sigma_0}$ by the projection formula. The structural sheaf $\Oo_{n\sigma_0} \cong \Oo_A/\Ii_C^n$ decomposes, as an $\Oo_C$-module, into $\Oo_C \oplus (\Ii_C/\Ii_C^2) \oplus \dots \oplus (\Ii_C^{n-1}/\Ii_C^{n})$. The intersection of $\sigma_0$ with the canonical divisor is zero, so $\Ii_C/\Ii_C^2 \cong \Oo_{\LL}(-\sigma_0)|_{\sigma_0} \cong L^{-1}$. Therefore,
\begin{equation} \label{eq identification of Hitchin bases}
|n\sigma_0|_{\supp \cap \sigma_\infty  = \emptyset} \cong \bigoplus_{i = 1}^n H^0(C,L^{\otimes i}).
\end{equation}

The canonical line bundle of $\LL$ is 
\begin{equation} \label{eq canonical line bundle of LL}
K_\LL \cong \Oo_{\LL}(-2 \sigma_\infty) \otimes p^* K_C \otimes p^*L^{-1},
\end{equation}
with canonical class
\[
[K_\LL] = -2 \sigma_\infty + (2g-2 -\ell) F.
\]
If $\ell > 2g_C-2$ or $ L = K_C$, the inverse of the canonical $K_\LL^{-1} \cong \bigwedge^2 \Tt \, \LL$ always has non-zero sections and we can equip $\LL$ with a Poisson structure by picking a Poisson bi-vector given by a non-zero section $\Theta_\LL \in H^0(\LL, K_\LL^{-1})$. In general, the Poisson structure of $\LL$ is not uniquely defined.

\subsection{Twisted Higgs bundles and their spectral data}
\label{sc L-Higgs bundles}

In our last preliminary section we introduce ($L$-twisted) Higgs bundles, whose spectral data are pure dimension 
$1$ sheaves on ruled surfaces.


As before, $C$ denotes a smooth projective curve and $L$ a line bundle on it. An {\it $L$-twisted Higgs bundle} over $C$ is a pair $(E, \varphi)$, where $E$ is a holomorphic vector bundle on $C$, and $\varphi \in H^0(C, \End(E) \otimes L)$ is a holomorphic section of the endomorphisms bundle, twisted by $L$. When $L$ is the canonical bundle $K_C$ of the curve, we refer to $K_C$-Higgs bundles, simply, as Higgs bundles. An $L$-twisted Higgs bundle $(E,\varphi)$ is \emph{(semi)stable} if every $\varphi$-invariant subbundle $F\subset E$ satisfies 
\[
\frac{\deg F}{\rk F} < ~(\leq)~ \frac{\deg E}{\rk E}.
\]
A semistable $L$-twisted Higgs bundle $(E,\varphi)$ is \emph{polystable} if it is a direct sum of stable bundles $(E,\varphi)= \bigoplus (E_i, \varphi_i)$ (all with the same slope $\deg E_i/\rk E_i$). It is possible to construct \cite{nitsure} the moduli space of rank $n$ and degree $d$ semistable (resp. stable) $L$-twisted Higgs bundles on $C$ which we denote by $\Mm^L_C(n,d)$ (resp. $\Mm^L_C(n,d)^{\st}$). When $L=K_C$, we will write $\Mm_C(n,d)$ (resp. $\Mm_C(n,d)^{\st}$) for the moduli space of semistable (resp. stable) Higgs bundles whose construction follows from \cite{hitchin-self, simpson2}, being a connected, normal and irreducible variety of dimension $2n^2(g_C-1) + 2$ \cite{simpson2}. When $\deg(L)>\deg(K)$ or $\deg(L)=\deg(K_C)$ but $L \ncong K_C$, the moduli space $\Mm^L_C(n,d)$ is a quasi-projective variety of dimension $2n^2\deg(L) + 1$ \cite{nitsure}.

Given a Higgs bundle $(E,\varphi)$, we see that $\varphi : E \to E \otimes L$ determines uniquely an action of $\Sym^\bullet(L^*)$ on $E$, the Higgs bundle $(E,\varphi)$ can be seen as a $\Sym^\bullet(L^*)$-module. Since $p$ is an affine morphism and $\Sym^\bullet(L^*) \cong p_*\Oo_{\Tot(L)}$, the push-forward under $p$ provides an equivalence of categories between $\Sym^\bullet(L^*)$-modules (including Higgs bundles) and $\Oo_{\Tot(L)}$-modules, called the {\it spectral correspondence} \cite{hitchin_duke, BNR} where the Higgs bundle $(E,\varphi)$ is sent to the pure dimension $1$ sheaf $\Ee$ on $\Tot(L)$ having rank $1$ on each irreducible component of its support, called the {\it spectral datum} of $(E,\varphi)$, and defined by $\ker(p^*\varphi - \tau)$, where $\tau$ denotes the tautological section of the pullback bundle $p^*L \to \Tot(L)$. One can observe that the spectral datum $\Ee$ associated to any Higgs bundle $(E,\varphi)$ is supported on a curve in the linear system $|n\sigma_0|$, that we call the {\it spectral curve} of $(E,\varphi)$, and the restriction of $\Ee$ to any irreducible component of the spectral curve has rank $1$, and over the whole spectral curve is a sheaf with degree
\begin{equation} \label{eq value of a}
a = d + \delta,    
\end{equation}
where we denote 
\begin{equation} \label{eq definition of delta}
\delta := \frac{1}{2}(n^2 - n)\ell, 
\end{equation}
denoting by $\ell$ the topological invariant $\ell$ of the ruled surface $\LL$. It then follows from \eqref{eq value of a} that the spectral datum of a $L$-Higgs bundle of rank $n$ and degree $d$ is a sheaf over $\LL$ with Mukai vector $\v_{d + \delta}$, where $C = \sigma_0$ on its definition. Starting from such $\Ee$, one obtains a Higgs bundle $(E,\varphi)$, by setting $E = p_*\Ee$ and $\varphi = p_* (\id_\Ee \otimes \tau)$, where we have considered the morphism given by tensor product with the tautological bundle,
\[
\id_\Ee \otimes \tau : \Ee \to \Ee \otimes p^*L.
\]
After \cite{schaub} one has that the stability notions of a Higgs bundle and the corresponding spectral data coincide. Therefore, the Higgs moduli space is the closed and dense subset of the moduli space of sheaves on $\LL$ with Mukai vector $\v_{d + \delta}$, where we recall \eqref{eq value of a}, given by those sheaves whose restriction to its associated spectral curve has rank $1$ on each irreducible component of its support,
\begin{equation} \label{eq compactification of Higgs space}
\Mm^L_C(n,d) \subset \left. \M_{\LL, H_0}(\v_{d + \delta}) \right|_{\supp \cap \sigma_\infty = \emptyset}.
\end{equation}

We have seen that $\LL$ is equipped with a Poisson structure $\Theta_\LL$ when $\ell > 2g_C-2$ or $ L = K_C$. In that case, one can define the Bottaccin--Markman Poisson structure on $\M_{\LL, H_0}(\v_{d + \delta})$, as we saw in Theorem \ref{tm Bottaccin-Markman Poisson str}, which restricts to a Poisson structure $\Theta_0$ on $\Mm^L_C(n,d)$. When $L = K_C$, $\Theta_0$ defines a symplectic form $\Omega_0$ on $\Mm_C(n,d)$. In the recent work \cite{biswas&bottaccin&gomez}, Biswas--Bottaccin--Gomez showed that this restriction coincides (up to scaling) with that obtained by extending the canonical symplectic form defined on the cotangent of the moduli space of stable vector bundles.

Since any divisor with null intersection with $\sigma_\infty$ is a multiple of $\sigma_0$, the restriction of \eqref{eq LePoitier} provides a morphism,
\begin{equation} \label{eq Hitchin system 2}
\Mm^L_C(n,d) \longrightarrow |n\sigma_0|_{\supp \cap \sigma_\infty = \emptyset},
\end{equation}
known as the {\it $L$-Hitchin fibration}, which takes the form
\begin{equation} \label{eq Hitchin system}
\Mm^L_C(n,d) \longrightarrow \Bb^L := \bigoplus_{i = 1}^n H^0(C,L^{\otimes i}),
\end{equation}
after the identification \ref{eq identification of Hitchin bases}. When $L=K_C$, this provides an algebraic completely integrable system \cite{hitchin_duke} called the {\it Hitchin system}.

\section{Involutions on Higgs moduli spaces from involutions on ruled surfaces}

\label{sc involutions and Higgs bundles}

\subsection{Some involutions on ruled surfaces}

\label{sc involutions on ruled surfaces}

We say that an involution on a variety equipped with a Poisson structure is Poisson if it preserves the Poisson structure under pull-back, and we call it anti-Poisson if the pull-back returns the inverse of the Poisson structure. In this section we study some involutions on $\LL$ and their relation with the Poisson structure.

The first involution we consider is
\begin{equation} \label{eq definition of -1_L}
\imath : \LL \to \LL,
\end{equation}
defined by multiplying by $-1$ the fibres of $L$. Note that its fixed locus is the disjoint union of the zero and infinity sections,
\begin{equation} \label{eq fixed locus on imath}
\Fix(\imath) = \sigma_\infty \sqcup \sigma_0.
\end{equation}

\begin{lemma} \label{lm imath_LL}
Every line bundle $J \in \Pic(\LL)$ is preserved by the involution $\imath$, {\it i.e.}
\[
\imath^*J \cong J.
\]
Suppose further that $\ell \geq 2g_C - 2$, then $\imath$ is an anti-Poisson involution with respect to any Poisson structure $\Theta_\LL$ defined on the ruled surface.
\end{lemma} 

\begin{proof}
One naturally has that $\imath^*\Oo_\LL(\sigma_\infty) \cong \Oo_\LL(\imath(\sigma_\infty))$, hence $\Oo_\LL(\sigma_\infty)$ is preserved by $\imath$ after \eqref{eq fixed locus on imath}. As $\imath$ preserves the fibres of $p : \LL \to C$, every line bundle in $p^*\Pic(C)$ is $\imath$-invariant and the first statement follows from \eqref{eq Pic of KK}. In particular, $K_\LL^{-1}$ is invariant.

As $\imath$ is an automorphism of $\LL$, we have that $d \imath \wedge d \imath$ provides an isomorphism between $\imath^* K_\LL^{-1}$ and $K_\LL^{-1}$. Since the eigenvalues of $d \imath$ are constantly $1$ and $-1$, one has  
\[
d\imath \wedge d \imath = -1,
\]
and all the sections of $K_\LL^{-1}$ are $\imath$-anti-invariant. This proves the second statement.
\end{proof}

Consider an involution on the base curve
\[
\zeta_C : C \to C.
\]
We say that a line bundle $L$ on $C$ is $\zeta_C$-invariant if there exists an isomorphism 
\[
f : \zeta_C^*L \stackrel{\cong}{\to} L.
\]
In this case, the action of $\zeta_C$ lifts, via $f$, to an action on $L$ that can be easily extended to its completion $\LL$,
\begin{equation} \label{eq definition of eta_C*}
\zeta_{f} : \LL \to \LL. 
\end{equation}
Observe that the involution associated to $-f$ coincides with the composition $\imath \circ \zeta_f$. Denote the associated quotient map by
\begin{equation} \label{eq definition of q_LL}
\q_f : \LL \to \quotient{\LL}{\zeta_f}.
\end{equation}

\begin{remark} \label{rm descent}
Suppose that $\q_C : C \to D:= \quotient{C}{\zeta_C}$ gives a smooth curve. By Kempf's descent lemma \cite{kempf}, if $\zeta_f$ acts trivially on the fibres of the points fixed under $\zeta_C$, the line bundle $L$ descends to a line bundle $W$ on $D$. Let us denote the associated ruled surface by $\WW := \PP \left ( \Oo_D \oplus W \right )$. Under these conditions, the quotient is
\[
\quotient{\LL}{\zeta_f} \cong \WW.
\]
\end{remark}

\begin{lemma} \label{lm eta_LL}
The involutions $\imath$ and $\zeta_f$ preserve cohomology class of every line bundle on $\LL$. Therefore, every Mukai vector on $\LL$ is $\imath$-invariant and $\zeta_f$-invariant. 
\end{lemma} 

\begin{proof}
After Lemma \ref{lm imath_LL} it is enough to prove the lemma for $\zeta_f$. Pull-back under $\zeta_f$ preserves the line bundle $\Oo_\LL(\sigma_\infty)$ and those line bundles in $p^*\Pic(\LL)$ coming from $\zeta_C$-invariant line bundles. Even if not every line bundle in $p^*\Pic(\LL)$ is preserved by $\zeta_f$, the involution preserves the class of a fibre, $F$. Then, the first statement follows from \eqref{eq Pic of KK}. The second follows immediately from the first.
%
\end{proof}

As $\zeta_C$ is an automorphism of $C$, $\partial \zeta_C$ provides a natural isomorphism between $K_C$ and $\zeta_C^*K_C$. Its compactification, $\KK = \PP(\Oo_C \oplus K_C)$ is a ruled surface of topological invariant $2g_C -2$. Let us denote by 
\begin{equation} \label{eq eta_KK plus}
\zeta_\KK^+ : = \zeta_{\partial \zeta_C},
\end{equation}
and
\begin{equation} \label{eq eta_KK minus}
\zeta_\KK^- : = \zeta_{-\partial \zeta_C}.
\end{equation}

After \eqref{eq canonical line bundle of LL} the canonical bundle of $\KK$ is
\[
K_\KK \cong \Oo_\KK(-2\sigma_\infty),
\]
which trivializes on $\Tot(K_C)$, the complement of $\sigma_\infty$. Recalling \eqref{eq multiples of sigma infty}, one has a unique (up to scaling) Poisson structure $\Theta_\KK$ on $\KK$ defining a canonical symplectic structure on $\Tot(K_C)$. 

\begin{lemma} \label{lm eta_KK}
The involution $\zeta_\KK^+$ is Poisson while $\zeta_\KK^-$ is anti-Poisson. Their restriction to $\Tot(K_C)$ provide, respectively, a symplectic and an anti-symplectic involution.

Suppose that $\q_C : C \to D := C/\zeta_C$ gives a smooth curve. Then, the quotient $\KK/ \zeta_{\KK}^-$ is smooth.
\end{lemma}

\begin{proof}
For the first statement, thanks to Lemma \ref{lm imath_LL}, it is enough to prove that $\zeta_\KK^+$, is a Poisson involution as $\zeta_\KK^-$ is the composition of $\zeta_\KK^+$ with the anti-Poisson involution $\imath$. Suppose first that $\q_C$ is unramified, one can easily show that the eigenvalues of $d \zeta_\KK^+$ are $d\zeta_C$ and $\partial \zeta_C$, hence, $d \zeta_C \wedge \partial \zeta_C = 1$ and $\zeta_\KK^+ = \zeta_{\partial \zeta_C}$ is Poisson.

After Remark \ref{rm descent}, the second statement follows from the observation that $\zeta_{\KK}^-$ acts with eigenvalues $d\zeta_C$ and $-\partial \zeta_C$. On the points of $\Fix(\zeta_C)$, one has that $d\zeta_C = -1$ so $-\partial \zeta_C= 1$ and $\zeta_\KK^-$ acts trivially on the fibres of the line bundle $K_C$ over $\Fix(\zeta_C)$. The statement then follows from Remark \ref{rm descent}.
\end{proof}

\subsection{Natural involutions under the spectral correspondence}

\label{sc natural involution under spectral correspondence}

Starting from the involutions $\imath$ and $\zeta_f$ on $\LL$ defined in Section \ref{sc ruled surfaces}, one can consider the natural involutions $\hat{\imath}$ and $\widehat{\zeta}_f$ on the moduli space of sheaves $\M_{\LL, H_0}(\v_{d + \delta})$, as defined in Section \ref{sc natural involutions}. Also, we recall the involutions $\xi_{J,\LL}$ and $\lambda_{J,f} = \xi_{J,\LL} \circ \widehat{\zeta}_f$ defined in full generality in Section \ref{sc Prymian fibrations}. In this section we describe the counterpart of these involutions on the moduli space of $L$--Higgs bundles. We finish this section studying the particular case of $L = K_C$, associated to Higgs bundles (with no twisting). We find in this case that the previous involtions give rise to well known involutions that have been widely studied \cite{heller&schaposnik, garciaprada&ramanan, basu&garciaprada, garciaprada&wilkins}. 

\begin{lemma} \label{lm involutions and spectral correspondence}
Under the spectral correspondence, the involution $\hat{\imath}$ corresponds to the biregular involution
\begin{equation}
\morph{\Mm^L_C(n,d)}{\Mm^L_C(n,d)}{(E,\varphi)}{(E,-\varphi).}{}{\check{\imath}}
\end{equation}
If, given the involution $\zeta_C : C \to C$, $f$ provides an isomorphism between $L$ and $\zeta_C^*L$, then $\widehat{\zeta}_{f}$ corresponds under the spectral correspondence to the involution
\begin{equation} \label{eq description of hat eta}
\morph{\Mm^L_C(n,d)}{\Mm^L_C(n,d)}{(E,\varphi)}{(\zeta_C^*E,(\id_{\zeta_C^*E} \otimes f) \circ \zeta_C^*\varphi),}{}{\check{\zeta}_f}
\end{equation}
which is biregular.
\end{lemma}

\begin{proof}
We begin by observing that $\check{\imath}$ and $\check{\zeta}_f$ preserve (semi)stability of Higgs bundles, so both a regular morphisms within the moduli of Higgs bundles. Recall that the spectral correspondence is realized by taking the push-forward under $p: \LL \to C$, obtaining the Higgs field by pushing-forward the morphism given by tensoring by the tautological section $\tau$, $\varphi = p_* (\id_\Ee \otimes \tau)$. In the first case, observe that $p \circ \imath = p$, so $\check{\imath}(E) = E$ and   
\[
\check{\imath}(\varphi) = p_*\left  (\id_{(\imath_{\LL})^*\Ee} \otimes \tau \right ) = p_* \imath_{\LL,*} \left  (\id_{(\imath_{\LL})^*\Ee} \otimes \tau \right ) = p_* (\id_\Ee \otimes (-\tau)) = -\varphi,
\]
what concludes the proof of the first statement. For the second statement note that $p \circ \zeta_f = \zeta_C \circ p$, which implies that $\check{\zeta}_f(E) = \zeta_C^*E$. Then, 
\[
\zeta_{C,*} \check{\zeta}_f(\varphi) = \zeta_{C,*} p_* \left ( \id_{\zeta_f^*\Ee} \otimes \tau \right ) = p_* \zeta_{f,*} \left ( \id_{\zeta_f^*\Ee} \otimes \tau \right ) =  p_* \left ( \id_{\Ee} \otimes \zeta_f^*\tau \right )
\]
Note also that the tautological section satisfies that 
\[
\tau = p^*f^{-1} \circ \zeta_f^*\tau,
\]
so,
\[
\zeta_{C,*} \check{\zeta}_f(\varphi) =  p_* \left ( \left ( \id_{\Ee} \otimes p^*f \right ) \circ \left ( \id_{\Ee} \otimes \tau \right ) \right ) = \left ( \id_{E} \otimes f \right ) \circ \varphi,
\]
from which the proof follows.
\end{proof}

Let us denote the fixed point locus of $\check{\zeta}_f$ by 
\[
\Nn_C^L(n,d) := \Fix(\check{\zeta}_f),
\]
which is, naturally, a closed subvariety of $\Mm_C^L(n,d)$.

As we saw in Remark \ref{rm descent}, Kempf descent lemma says that the line bundle $L$ descends to a line bundle $W$ on $D$ if the action of $f$ is trivial on the ramification locus of $\q_C: C \to D$. In that case, the quotient $\LL / \zeta_f$ is the ruled surface $\WW$. Observe, as well, that in that case $L \cong \q_C^*W$.

\begin{lemma} \label{lm Nn supported on vector space}
Under the spectral correspondence, $\Nn_C^L(n,d)$ embedds into $\N_{\LL, H_0}(\v_{d + \delta}, nD)$. Furthermore, the Hitchin fibration restricts there to
\begin{equation} \label{eq Hitchin fibration for N} 
\Nn_C^L(n,d) \to  \bigoplus_{i = 1}^n \q_C^*H^0(D, W^{\otimes i}).
\end{equation}
\end{lemma}

\begin{proof}
We recall from \eqref{eq Hitchin system 2} that the image of the support (Hitchin) morphism is contained in (an open subset of) a linear system $|n\sigma_0|$ in $\LL$. The fixed locus under the action of $\widehat{\zeta}_f$ on a linear system is disconnected and given by the union of the projectivization of the $+1$ and $-1$ eigenspaces of the associated space of sections. Observe that the component associated to $+1$ coincides with $\q_\LL^*|nD|$, containing the open subset $\bigoplus_{i = 1}^n \q_C^*H^0(D, W^{\otimes i})$. The rest of the proof follows easily from this observation.
\end{proof}

One can consider the morphism
\begin{equation} \label{eq definition of hatq HB}
\morph{\Mm^W_D(n,d')}{\Nn^L_C(n,2d') \subset \Mm^L_C(n,2d')}{(E,\varphi)}{(\q_C^*E, \q_C^*\varphi),}{}{\widehat{\q}_C}
\end{equation}
which corresponds to \eqref{eq definition of hatq} under the spectral correspondence. If the spectral curves in $\Tot(L)$ are ramified over the spectral curves in $\Tot(W)$, then, Proposition \ref{pr q embedding on ssl} implies that $\widehat{\q}_C$ is generically 1:1, being an isomorphism when we restrict to the locus of smooth curves over smooth curves (corresponding to the locus indexed by $\ssl$). If, on the contrary, $\zeta_C$ is unramified, Proposition \ref{pr q embedding on ssl} says that $\widehat{\q}_C$ factors through
\[
\xymatrix{
\Mm^W_D(n,d') \ar[rr]^{\widehat{\q}_S} \ar[rd] & & \Nn^L_C(n,2d') \subset \Mm^L_C(n,2d')
\\
 & \quotient{\Mm^W_D(n,d')}{\ZZ_2}, \ar[ur]_{\quad \textnormal{generically} \, 1:1} & 
}
\]
where $\ZZ_2$ acts under tensorization by the $2$-torsion line bundle on $D$ associated to the unramified $2:1$ cover $\q_C : C \to D$.

Suppose one can pick an $\zeta_f$-invariant line bundle $J$ such that condition \eqref{eq condition for J and v} holds with respect to $\v_{d + \delta}$. Observe that, after \eqref{eq Pic of KK} and since $\sigma_\infty \cdot \sigma_0 = 0$, one has that the restriction of $J$ to any curve $A$ on $|n \sigma_0|$ is
\begin{equation} \label{eq description of J at C}
J|_{A} \cong p^* J_C,
\end{equation}
for some $J_C \in \Pic(C)$.

We now consider the involution $\xi_{J,\LL}$ and the composition $\lambda_{J_0,f} = \xi_{J,\LL} \circ \widehat{\zeta}_f$ which is another involution as $\xi_{J,\LL}$ and $\widehat{\zeta}_f$ commute due to the $\zeta_f$-invariance of $J$. Let us also describe the counterpart of these involutions in $\Mm^L_C(n,d)$. 

\begin{lemma} \label{lm involutions and spectral correspondence 2}
Under the spectral correspondence, the involution $\xi_{J,\LL}$ corresponds to
\begin{equation}
\morph{\Mm^L_C(n,d)}{\Mm^L_C(n,d)}{(E,\varphi)}{(E^*\otimes J_C L,\varphi^t).}{}{\check{\xi}_{J_C,L}}
\end{equation}
Also, $\lambda_{J,f}$ corresponds to
\begin{equation} \label{eq lambda}
\morph{\Mm^L_C(n,d)}{\Mm^L_C(n,d)}{(E,\varphi)}{(\zeta_C^*E^* \otimes J_C L,(\id_{\zeta_C^*E} \otimes f)^t \circ \zeta_C^*\varphi^t),}{}{\check{\lambda}_{J_C,f}}
\end{equation}
both being biregular involutions. 
\end{lemma}

\begin{proof}
After \eqref{eq fibrewise description of xi}, 
\[
\check{\xi}_{J_C,L}(E) = p_* \Hh \! om_{A}(\Ff|_{A}, J(A)|_{A}) \cong p_* \Hh \! om_{A}(\Ff|_{A}, p^*(J_C L^{n})),
\]
where, in this identification, we have made use of \eqref{description of self intersection of sigma_0} and \eqref{eq description of J at C}. 
Thanks to \eqref{eq canonical line bundle of LL}, we observe that the relative canonical sheaf $\omega_{A/C}$ is $p^*L^{n-1}$. Hence, 
\[
\check{\xi}_{J,\LL}(E) \cong p_* \Hh \! om_{A}(\Ff|_{A}, p^*(J_C L) \otimes \omega_{A/C}) = p_* \Hh \! om_{A}(\Ff|_{A}, p^!(J_C L)) \cong E^* \otimes J_C L,
\]
and the first statement follows. The second statement is a direct consequence of the first and Lemma \ref{lm involutions and spectral correspondence}. 
\end{proof}

\begin{remark}
Recall from Remark \ref{rm choice of J_0 and v_0} that, choosing $\v_0$ as Mukai vector and $J_0 = \Oo_\LL(-n\sigma_0)$, one gets $\xi_{J_0,L}$ is the dualizing involution and
\[
\morph{\Mm^L_C(n,-\delta)}{\Mm^L_C(n,-\delta)}{(E,\varphi)}{(E^*\otimes L^{1-n},\varphi^t).}{}{\check{\xi}_{J_0,L}}
\]
\end{remark}

\begin{remark}
For $\delta$ as in \eqref{eq definition of delta}, the pair of Mukai vector $\v_\delta$ and line bundle $J_\delta = \Oo_\LL(-\sigma_0)$ satisfies condition \eqref{eq condition for J and v}. In that case, one gets
\[
\morph{\Mm^L_C(n,0)}{\Mm^L_C(n,0)}{(E,\varphi)}{(E^*,\varphi^t).}{}{\check{\xi}_{J_{\delta,C},L}}
\]
\end{remark}

Let us denote the fixed point locus of $\check{\lambda}_{J_C, L}$ by
\[
\Pp_C^L(n,d,J_C) := \Fix(\check{\lambda}_{J_C,
L}).
\]
Observe that $\Nn_C^L(n,d)$ and $\Pp_C^L(n,d, J)$ are given by the restriction of $\N_{\LL, H_0}(\v_{d + \delta})$ and $\P_{\LL, H_0}(\v_{d + \delta},J)$ in \eqref{eq compactification of Higgs space}. Recalling the subvarieties $\N_{\LL, H_0}(\v_{d + \delta}, nD)$ and $\P_{\LL, H_0}(\v_{d + \delta},J, nD)$ from \eqref{eq N over B} and \eqref{eq P over B}, respectively, we study their relation with $\Nn_C^L(n,d)$ and $\Pp_C^L(n,d, J)$.

\begin{lemma} \label{lm Pp supported on vector space}
Under the spectral correspondence, $\Pp_C^L(n,d, J_C)$ embedds into $\P_{\LL, H_0}(\v_{d + \delta},J, nD)$~. Furthermore, the Hitchin fibration restricts to
\begin{equation} \label{eq Hitchin fibration for P}
\Pp_C^L(n,d, J_C) \to  \bigoplus_{i = 1}^n \q_C^*H^0(D, W^{\otimes i}),
\end{equation}
where $W$ is described in Remark \ref{rm descent}. 
\end{lemma}

\begin{proof}
Mutatis mutandis, the proof follows as the proof of Lemma \ref{lm Nn supported on vector space}.
\end{proof}

Sometimes in the literature, in the particular case of $L = K_C$ and $f = \partial \zeta_C$, one abbreviates $(\id_{\zeta_C^*E} \otimes \partial \zeta_C) \circ \zeta_C^*\varphi$ simply by $\zeta_C^*\varphi$. Therefore, the natural involution associated to  $\zeta_\KK^+ = \zeta_{\partial \zeta_C}$, as described in \eqref{eq eta_KK plus} and \eqref{eq eta_KK minus}, would be expressed as
\begin{equation} \label{eq eta K}
\morph{\Mm_C(n,d)}{\Mm_C(n,d)}{(E,\varphi)}{(\zeta_C^*E, \pm\zeta_C^*\varphi),}{}{\check{\zeta}_K^\pm}
\end{equation}
under the spectral correspondence. This immediately implies that the counterpart of $\lambda_{J, \KK}^\pm = \xi_{J,\KK} \circ \widehat{\zeta}_K^\pm$ on the moduli space of Higgs bundles is
\begin{equation} \label{eq lambda K}
\morph{\Mm_C(n,d)}{\Mm_C(n,d)}{(E,\varphi)}{(\zeta_C^*E^* \otimes J_CK_C, \pm\zeta_C^*\varphi^t).}{}{\check{\lambda}_{J_C,K}^\pm}
\end{equation}

The involutions $\check{\zeta}_K^\pm$, along with their fixed point locus, have been studied by Heller--Schaposnik \cite{heller&schaposnik} and Garcia-Prada--Wilkins \cite{garciaprada&wilkins}. The involutions $\check{\lambda}_{J, K}^\pm$ are covered by the general study of Basu and Garcia-Prada \cite{basu&garciaprada} of automorphisms of the moduli space of Higgs bundles. We combine their results in the following corollary, that could also derived from Propositions \ref{pr natural involutions preserve behaviour under Poisson structure} and \ref{pr eta antisymplectic}.

\begin{corollary}
The involution $\check{\zeta}_{K}^+$ is symplectic and its fixed locus $\Nn_C^+(n,d)$ is a holomorphic symplectic subvariety of the Higgs moduli space $\Mm_C(n,d)$, while $\check{\zeta}_{K}^-$ is anti-symplectic, and defines a Lagrangian subvariety $\Nn_C^-(n,d)$ of $\Mm_C(n,d)$. 

Similarly, $\check{\lambda}_{J_C, K}^+$ is an anti-symplectic involution and its fixed locus $\Pp_C^+(n,d, J_C)$ a Lagrangian subvariety of $\Mm_C(n,d)$, while $\check{\lambda}_{J_C,K}^-$ is symplectic and fixes a holomorphic symplectic subvariety $\Pp_C^-(n,d, J_C)$. 
\end{corollary}

\section{Non-linear degeneration of integrable systems}

\label{sc degenerations}

\subsection{Extending the non-linear degeneration of Donagi, Ein and Lazarsfeld}

\label{sc degeneration of DEL}

In this section we generalize the Donagi, Ein and Lazarsfeld \cite{donagi&ein&lazarsfeld} non-linear degeneration of the moduli space of pure dimension $1$ sheaves on a K3 surface into the moduli space of Higgs bundles over a curve inside the K3. Such degeneration was recently generalized to the case of curves in abelian surfaces by de Cataldo, Maulik and Shen \cite{deCataldo&maulik&shen_1}, being a central step on their proof of the P=W conjecture over genus $2$ curves. 
In this section we provide a degeneration of the moduli space of pure dimension $1$ sheaves over any smooth surface, into the moduli space of $L$-Higgs bundles over a smooth subcurve, being $L$ the normal bundle of our curve inside the surface.

Take a smooth surface $S$ and consider a smooth curve $C \subset S$ inside it of genus $g_C$. From now on, let us denote its normal bundle by 
\begin{equation} \label{eq setting L}
L := \Oo_S(C)|_C,
\end{equation}
that coincides with $K_C \otimes (K_S|_C)^{-1}$ by adjunction. The projective completion of the normal cone is isomorphic to the ruled surface $\LL$ defined in \eqref{eq definition of LL}. This implies that the canonical divisor of $\LL$ is
\[
K_\LL = -2 \sigma_\infty + p^*K_C - p^*\Oo_S(C)|_C,
\]
By the genus formula, and recalling that $\sigma_0$ does not intersect with $\sigma_\infty$, we get that
\begin{equation} \label{eq equality of self intersections}
\sigma_0^2 = \ell = C^2
\end{equation}
and
\[
\sigma_0 \cdot K_\LL = C \cdot K_S.
\]
Hence, the linear systems $|nC|$ and $|n\sigma_0|$ contain generic curves of equal genus,
\begin{equation} \label{eq equal genus}
g_{nC} = g_{n\sigma_0},
\end{equation}
and the topological invariants of the ideal sheaves describing these curves coincide.

Take the trivial family $S \times \PP^1$ and consider the blow-up at $C \times \{ 0 \}$,
\begin{equation} \label{eq definition of Ss}
\ol{\sS} := \Blow_{C \times \{ 0 \} }\left (S \times \PP^1 \right ),
\end{equation}
which is non-singular as we are blowing-up a non-singular subvariety. Composing the structural morphism of the blow-up with the projection $S \times \PP^1 \to \PP^1$, one gets the structural morphism 
\[
\pi : \ol{\sS} \to \PP^1. 
\]
This map is flat and its fibres outside $0 \in \PP^1$ are not affected by our construction, so for the generic fibre, at $t \neq 0$, we have
\[
\ol{\sS}_t \cong S.
\]
Meanwhile, the central fibre, at $t = 0$, is the union
\begin{equation} \label{eq definition of Ss_0}
\ol{\sS}_0 = \LL \cup S,
\end{equation}
where both irreducible components meet transversely at the image of infinity section $\sigma_\infty \subset \LL$ identified with the curve $C \subset S$. Note that $\ol{\sS}_0$ is a complete intersection subvariety. 
One also consider the open subset
\begin{equation} \label{eq definition of Ss'}
\sS := \ol{\sS} - ( S \times  \{ 0 \}),
\end{equation}
and observe that the restriction there of the structural morphism gives, again, a flat family of surfaces over $\PP^1$, but now the central fibre $\sS|_{t = 0} = \Tot(L)$ is not projective.

Denote by $p_S$ the composition of the structural morphism of the blow-up with the projection of $S \times \PP^1$ onto $S$. Being $\ol{\sS}$ projective, one may choose a relative polarization
\begin{equation} \label{eq definition of Hh}
\hH = p_S^* \H \otimes \Oo_{\ol{\sS}}(\LL). 
\end{equation}
For all $t \neq 0$ in $\PP^1$, one trivially has that $\hH|_{\ol{\sS}_t} = \H$ while $H_0 := \hH|_{\LL}$ corresponds to $p^*(\H|_C) \otimes \Oo_{\LL}(\sigma_{\infty})$.

Since $C \times \{ 0 \}$ is a Cartier divisor of $C \times \PP^1$, one has that 
\begin{equation} \label{eq definition of Cc}
\cC := \Blow_{C \times \{ 0 \}}\left ( C \times \PP^1 \right ) \cong C \times \PP^1
\end{equation}
embeds naturally into $\ol{\sS}$. One can consider the fibration $\cC \to \PP^1$, whose generic fibres, at $t\neq 0$, are naturally identified with $C$,
\[
\cC_t \cong C \subset S = \ol{\sS}_t.
\]
Furthermore, the restriction of $\cC$ to $t = 0$ coincides with the image of the zero section,
\[
\cC_0 \cong \sigma_0 \cong C, 
\]
inside $\LL \subset \ol{\sS}_0$. Note that the restriction to every point of $\PP^1$ gives $\cC_t \in H^2(\ol{\sS}_t, \ZZ)$. Using $\cC$, consider the Mukai vector
\begin{equation} \label{eq definition of v}
\vV_a = \left ( 0, [n\cC], a - n^2\ell   \right ).
\end{equation}
Let us consider the relative Hilbert scheme parametrizing ideal sheaves of curves in $\ol{\sS}$ with first Chern class $[n\cC]$ and let us denote by $\{ n\cC \}$ the connected component containing the family of curves $n\cC$. Let us consider, inside it, the open subset given by those curves $A \subset \ol{\sS}_t$ contained in the $\sS \subset \ol{\sS}$,
\begin{equation} \label{eq definition of Bb}
\{ n\cC \}_\sS := \{ n\cC \} |_{A \subset \sS}.
\end{equation}

Take a point of $\{ n\cC \}_\sS$ representing the ideal sheaf $\Ii_A \hookrightarrow \Oo_{\sS_t}$ of a curve $A \subset \sS_t$. We see that $\Ii_A$ is of rank $1$ and pure dimension $2$ over a smooth surface $\sS_t$ which either is projective (for $t \neq 0$) or an affine subset of a smooth projective surface $\LL$ (for $t = 0$) where $\Ii_A$ extends naturally. It follows that the determinant in cohomology is well define and one can consider 
\begin{equation} \label{eq det_Ss}
\det_{\sS} : \{ n\cC \}_\sS \to \Jac_{\PP^1}(\ol{\sS}).
\end{equation}
 
For $t \neq 0$, the fibre of \eqref{eq det_Ss} over $L \in \Jac(S)$ corresponds to the linear system $|L|$ in $S$. We recall once again that every irreducible projective curve contained in $\Tot(L)$ can be deformed linearly to a multiple curve supported on $\sigma_0$. Then, the image of $\{ n\cC \}_\sS|_{t = 0}$ under \eqref{eq det_Ss} is $\Oo_{\ol{\sS}_0}(n \sigma_0)$ and  $\{ n\cC \}_\sS|_{t = 0}$ is the locus of $|n \sigma_0|$ inside $\LL$ which does not intersect $\sigma_\infty$. After \eqref{eq identification of Hitchin bases}, this is
\begin{equation} \label{eq description of Bb_0}
\{ n\cC \}_\sS|_{t=0} \cong \bigoplus_{i = 1}^n H^0(C, L^{\otimes i}),
\end{equation}
coinciding with the Hitchin base in \eqref{eq Hitchin system}.

Following Simpson \cite[Theorem 1.21]{simpson1} (see \cite[Theorem 4.3.7]{huybrechts&lehn} for a specific treatment of the relative case), let us consider the moduli space $\M_{\ol{\sS}/\PP^1, \hH}(\vV_a)$
of relatively pure dimension $1$ $\hH$-semistable relative sheaves over the $\PP^1$ scheme $\ol{\sS}$ having Mukai vector $\vV_a$. By construction, $\M_{\ol{\sS}/\PP^1, \hH}(\vV_a)$ comes equipped with the structural flat morphism to $\PP^1$, which trivializes over $\PP^1 - \{ 0 \}$ having generic fibres equal to $\M_{S,\H}(\v_a)$, and whose central fibre over $t = 0$ is $\M^{\hH_0}_{\LL \cup S}(\v_a)$. By specifying the relative (Fitting) support, one has the following surjective morphism 
\begin{equation} \label{eq relative Beauville fibration}
\morph{\M_{\ol{\sS}/\PP^1, \hH}(\vV_a)}{\Hilb_{\ol{\sS}/\PP^1}([n\cC])}{\Ff}{\supp(\Ff),}{}{\mathbf{h}}
\end{equation}
commuting with both structural morphisms. Consider the restriction $\M_{\ol{\sS}/\PP^1, \hH}(\vV_a) |_{\{ n\cC \}_\sS}$ having generic fibres $\M_{S,\H}(\v_a,nC)$ over $t \neq 0$ and central fibre $\M^{H_0}_{\LL}(\v_a)|_{\supp \cap \sigma_\infty = \emptyset}$ at $t = 0$. For the description of the central fibre, note that a sheaf in $\ol{\sS}_0 = \LL \cup S$ whose support fits in $\Tot(L) \subset \LL$ is $H_0$-semistable (resp. $H_0$-stable) as a sheaf in $\LL$ if and only if it is $\hH_0$-semistable (resp. $\hH_0$-stable) as a sheaf in the whole $\ol{\sS}_0$. Recalling \eqref{eq compactification of Higgs space}, we define the open subset
\begin{equation} \label{eq definition of Mm}
\mM_{\sS/\PP^1, \hH}(\vV_a,n\cC) \subset \M_{\ol{\sS}/\PP^1, \hH}(\vV_a) |_{\{ n\cC \}_\sS},
\end{equation}
given by the complement of those sheaves at $t = 0$ having rank different from $1$ on some irreducible component of their support.

We summarize all of the above in the following theorem. This is a generalization to arbitrary smooth surfaces of the non--linear deformation that appeared first in \cite{donagi&ein&lazarsfeld} for K3 surfaces and in \cite{deCataldo&maulik&shen_1} for abelian surfaces.

\begin{theorem}[\cite{donagi&ein&lazarsfeld}] \label{tm Mm_S}
Consider a smooth surface $S$ and a genus $g \geq 2$ smooth projective curve $C$ inside $S$. Denote the normal bundle of $C$ in $S$ by $L$. Choose a positive integer $n$ and a Mukai vector $\v_a$ as in \eqref{eq definition of v_a} constructed out of $n$ and $C7$. One can construct $\mM_{\sS/\PP^1, \hH}(\vV_a,n\cC)$ and $\{ n\cC \}_\sS$ flat over $\PP^1$, and a morphism between them commuting with the structural morphisms, 
\begin{equation} \label{eq commuting diagram Mm_S}
\xymatrix{
\mM_{\sS/\PP^1, \hH}(\vV_a,n\cC) \ar[rr] \ar[rd] & & \{ n\cC \}_\sS \ar[ld]
\\
& \PP^1. &
}
\end{equation}
The horizontal arrow is a surjective fibration, trivial over $\PP^1 - \{ 0 \}$, and the generic fibre over $t \neq 0$ is the Le Poitier morphism \eqref{eq LePoitier}, 
\[
\M_{S,\H}(\v_a,nC) \to \{ nC \},
\]
while the central fibre at $t = 0$ gives the Hitchin morphism \eqref{eq Hitchin system},
\[
\Mm^L_C(n,d) \to \bigoplus_{i = 1}^n H^0(C, L^{\otimes i}),
\]
where $d = a - \delta$.
\end{theorem}

\begin{remark} \label{rm chi-independence}
Since the normal bundle $L$ of a curve $C$ in a del Pezzo surface has $\deg(L) > \deg(K_C)$, we observe that the degeneration given in Theorem \ref{tm Mm_S} connects the two moduli spaces considered in the work of Maulik and Shen \cite{maulik&shen}, and for which cohomological $\chi$-independence is proven.
\end{remark}

\subsection{Deforming the symplectic structure}

\label{sc symplectic structure}

In this section we restrict ourselves to the case where $S$ is a symplectic surface. In this context we study the symplectic structure on the moduli spaces involved in the construction obtained in Section \ref{sc degeneration of DEL}, showing that it provides a deformation of symplectic varieties.

Since the canonical bundle of $S$ is trivial, one gets, by adjunction, that the normal bundle of $C \subset S$ is its canonical bundle $K$. Recall then that the family of surfaces $\ol{\sS} \to \PP^1$ has central fibre $\ol{\sS}_0 = \KK \cup S$ meeting at the identification of $\sigma_\infty \subset \KK$ with $C \subset S$. We start by studying $K_{\ol{\sS}/\PP^1}$ and $K_{\ol{\sS}_0}$. 

\begin{proposition} \label{pr relative canonical bundle}
Let $S$ be a symplectic surface and pick $C \subset S$ smooth curve of genus $g \geq 2$. One has the following
\begin{enumerate}

\item \label{it description of omega_Ss_0} $K_{\ol{\sS}_0}$ is the line bundle given by $\Oo_\KK(-\sigma_\infty)$ and $\Oo_S(C)$ identified along $\sigma_\infty \cong C$ by a natural isomorphism between $K \cong \Oo_\KK(-\sigma_\infty)|_{\sigma_\infty}$ and $K \cong \Oo_S(C)|_C$;

\item \label{it sections of omega} $H^0(\ol{\sS}_0, K_{\ol{\sS}_0}) = \cC$ and every non-trivial section of $K_{\ol{\sS}_0}$ vanishes on $\KK$ and only on $\KK$;

\item \label{it sections of omega^-1} $H^0(\ol{\sS}_0, \omega^{-1}_{\ol{\sS}_0}) = \cC$ and every non-trivial section of $\omega^{-1}_{\ol{\sS}_0}$ vanishes on $S$ and only on $S$; 

\item \label{it pi_* omega} $\pi_*K_{\ol{\sS}/\PP^1} \cong \Oo_{\PP^1}(1)$;

\item \label{it pi_* omega^-1} $\pi_*\omega^{-1}_{\ol{\sS}/\PP^1} \cong \Oo_{\PP^1}(-1)$.

\end{enumerate}
\end{proposition}

\begin{proof}
We first observe that $\ol{\sS}$ is constructed by blowing-up a smooth subvariety of $S \times \PP^1$, which is also smooth. Hence $\ol{\sS}$ is smooth and has canonical bundle $K_{\ol{\sS}} = r^*K_{S \times \PP^1}(\KK)$, where we denote by $r$ the structural morphism of the blow-up and we recall that $\KK$ is the exceptional divisor. One can deduce that the singular central fibre $\ol{\sS}_0 = \KK \cup S$, where both components meet transversally, is a complete intersection variety, so it is Gorenstein and its dualizing sheaf $K_{\ol{\sS}_0}$ is a line bundle. This implies that the relative dualizing sheaf $K_{\ol{\sS}/\PP^1}$ is a line bundle too, indeed it is $\Oo_{\ol{\sS}}(\KK)$, as $K_S$ is trivial. 

Since $\ol{\sS}_0$ is a complete intersection divisor in $\ol{\sS}$, one has that $K_{\ol{\sS}_0} = K_{\ol{\sS}/\PP^1}(\KK + S) |_{\ol{\sS}_0}$, where we abuse of notation by denoting the total transform of $S \times \{ 0 \}$ simply by $S$. Then, 
\[
K_{\ol{\sS}_0}|_\KK \cong \Oo_{\ol{\sS}}(2\KK)|_\KK \otimes \Oo_{\ol{\sS}}(S)|_\KK \cong \Oo_\KK(-2\sigma_\infty) \otimes \Oo_\KK(\sigma_\infty) \cong \Oo_\KK(-\sigma_\infty),
\] 
where we recall that $\Oo_{\ol{\sS}}(-\KK)|_\KK$ is the linearization $\Oo_\KK(1) \cong \Oo_\KK(\sigma_\infty)$. Similarly, the restriction of $K_{\ol{\sS}_0}$ to $S$ gives
\[
K_{\ol{\sS}_0}|_S \cong \Oo_{\ol{\sS}}(2\KK)|_S \otimes \Oo_{\ol{\sS}}(S)|_S \cong \Oo_{S}(2C) \otimes \Oo_{S}(-C) \cong \Oo_S(C),
\]
where we recall that $K_S \cong \Oo_{\ol{\sS}}(\KK + S)|_S \cong \Oo_{\ol{\sS}}(\KK)|_S \otimes \Oo_{\ol{\sS}}(S)|_S$ is trivial, hence $\Oo_{\ol{\sS}}(S)|_S \cong (\Oo_{\ol{\sS}}(\KK)|_S)^{-1}$. Recalling that $K_\KK = \Oo_\KK(-2\sigma_\infty)$ and $K_S \cong \Oo_S$, applying adjunction, one gets that $\Oo_\KK(-\sigma_\infty)|_{\sigma_\infty}$ and $\Oo_S(C)|_C$ both being isomorphic to $K$, the canonical bundle of $\sigma_\infty \cong C$, so the restriction of the line bundle $K_{\ol{\sS}_0}$ provides a natural identification. This finishes the proof of \eqref{it description of omega_Ss_0}. 

Since $\Oo_\KK(-\sigma_\infty)$ has no non-zero sections, every non-zero section of $K_{\ol{\sS}_0}$ vanishes completely on $\KK$, hence it is given by a section of $\Oo_S(C)$ vanishing identically at $C$. The set of those sections determines a $1$-dimensional subspace of $H^0(S,\Oo_S(C))$ and \eqref{it sections of omega} follows. 

We proof \eqref{it sections of omega^-1} by observing that $h^0(\KK, \Oo_\KK(\sigma_\infty)) = 1$ and that $h^0(S, \Oo_S(-C)) = 0$. It follows that every non-zero section of $K_{\ol{\sS}_0}^{-1}$ vanishes on $S$, so it is given by a section of $\Oo_\KK(\sigma_\infty)$ which vanishes identically at $\sigma_\infty$.

Recall that $\ol{\sS}|_{\PP^1 - \{ 0 \}}$ is the trivial fibration $S \times (\PP^1 - \{ 0 \})$, so  the restriction there of $K_{\ol{\sS}/\PP^1}$ is the trivial line bundle. After \eqref{it sections of omega} and \eqref{it sections of omega^-1} and upper semicontinuity of cohomology, one has that $\pi_*K_{\ol{\sS}/\PP^1}$ and $\pi_*\omega^{-1}_{\ol{\sS}/\PP^1}$ are line bundles over $\PP^1$, inverse to each other. Note that $K_{\ol{\sS}/\PP^1} \cong \Oo_{\ol{\sS}}(\KK)$ comes naturally equipped with a section. Furthermore, $h^0(\Oo_{\ol{\sS}}(\KK)) = 1$ by the properties of blow-up. It follows that $\pi_*K_{\ol{\sS}/\PP^1}$ has a single non-zero section (up to scaling), so \eqref{it pi_* omega} follows. 

Finally, \eqref{it pi_* omega^-1} follows from \eqref{it pi_* omega}.  
\end{proof}

In view of \eqref{it sections of omega} and \eqref{it sections of omega^-1} of Proposition \ref{pr relative canonical bundle} we shall first construct a Poisson structure on $\M_{\ol{\sS}/\AA^1, \hH}(\vV_a)$ and, then, derive the symplectic structure on $\mM_{\sS/\AA^1, \hH}(\vV_a,n\cC)$. Recall that $\ell = 2g_C - 2$ in this case, so we fix $d = a - (n^2-n)(g_C - 1)$.

\begin{theorem} \label{tm relative symplectic form}
Consider a smooth symplectic surface $S$ and a genus $g \geq 2$ smooth projective curve $C$ inside $S$. Then, there exists a relative Poisson structure $\Theta$ on $\M_{\ol{\sS}/\AA^1, \hH}(\vV_a) \to \AA^1$ which defines a relative symplectic form $\varOmega$ on $\mM_{\sS/\AA^1, \hH}(\vV_a,n\cC) \to \AA^1$ coinciding (up to scaling) with the Mukai form $\Omega$ over the generic fibre $\M_{S, \H}(\v_a)$ over $t \neq 0$, and, on the central fibre $\Mm_C(n,d)$ at $t = 0$, with $\Omega_0$ obtained by extending the canonical symplectic form on the cotangent of the moduli space of stable vector bundles.
\end{theorem}

\begin{proof}
Starting from \eqref{it pi_* omega^-1} of Proposition \ref{pr relative canonical bundle}, we pick $\AA^1 \subset \PP^1$ a trivialization of $\pi_*K_{\ol{\sS}/\PP^1}^{-1}$ containing $0 \in \PP^1$ and choose a section $\vartheta \in H^0(\AA^1, \pi_*K_{\ol{\sS}/\AA^1}^{-1})$. Following Theorem \ref{tm Bottaccin-Markman Poisson str} one can construct a relative Bottacin--Markman Poisson structure $\Theta$ on $\M_{\ol{\sS}/\AA^1, \hH}(\vV_a)$, 
\begin{align} \label{eq relative Bottacin form}
\Ext^1_{\Oo_{\ol{\sS}_t}}(\Ee, \Ee\otimes K_{\ol{\sS}_t}) \wedge \Ext^1_{\Oo_{\ol{\sS}_t}}(\Ee, \Ee \otimes K_{\ol{\sS}_t})  \stackrel{\circ}{\longrightarrow} & \Ext^2_{\Oo_{\ol{\sS}_t}}(\Ee, \Ee\otimes K_{\ol{\sS}_t}^2) \stackrel{\tr}{\longrightarrow}
\\
\nonumber
& H^{2}(\ol{\sS}_t,K_{\ol{\sS}_t}^2)  \stackrel{\langle \cdot , \vartheta \rangle}{\longrightarrow} H^{2}(\ol{\sS}_t,K_{\ol{\sS}_t}) \cong \cC.
\end{align}
Thanks to \eqref{it sections of omega^-1} of Proposition \ref{pr relative canonical bundle}, we can check that the restriction of $\Theta$ to the subset $\Mm_C(n,d)$ of the central fibre is non-degenerate, hence defines a symplectic form there, as the tangent and the cotangent spaces to the moduli can be identified thanks to a trivialization of $K_{\KK_C} = \Oo_{\KK_C}(-2\sigma_\infty)$ over $\Tot(K_C) = \KK_C - \{ \sigma_\infty \}$. Furthermore, the canonical bundle is trivial over the generic fibres, $K_S \cong \Oo_S$, so in this case the cotangent space of $\M_{\ol{\sS}/\AA^1, \hH}(\vV_a)|_t = \M_{S}^{\H}(\v_a)$ is identified with the tangent space. Then, the Poisson form $\Theta_t$ over $t \neq 0$ is non-degenerate and defines a symplectic form which coincides with the Mukai form up to scaling. Hence, the restriction of $\Theta$ to $\mM_{\sS/\AA^1, \hH}(\vV_a,n\cC)$ defines a relative symplectic form $\varOmega$.
\end{proof}

\subsection{Degenerating the fixed locus of involutions}
\label{sc degenerating involutions}

In this section, we study the behaviour under the Donagi--Ein--Lazarsfeld degeneration of the involutions considered in Sections \ref{sc natural involutions} and \ref{sc Prymian fibrations}. By doing so, we generalize to the case of an arbitrary smooth surface equipped with an involution, the construction of Sawon and Shen \cite{sawon&shen} who treated the case described in Section \ref{sc sawon and shen}.

Consider a smooth surface $S$ equipped with an arbitrary involution $\zeta_S : S \to S$. Pick a smooth curve $C$ preserved by $\zeta_S$ and denote by $\zeta_C: C \to C$ the restriction of the involution to it. As in \eqref{eq setting L}, denote the normal bundle by $L= \Oo_S(C)|_C$, and note that $d \zeta_S : \zeta_S^* \Tt S \stackrel{\cong}{\longrightarrow} \Tt S$ sends $\zeta_C^* \Tt C$ to $\Tt C$ inducing the isomorphism
\begin{equation} \label{eq definition of f_0}
f_0 := [d\zeta_S]_{\Tt C} : \zeta_C^*L = \zeta_S^* \left(\quotient{\Tt S}{\Tt C} \right)  \longrightarrow L = \quotient{\Tt S}{\Tt C}.
\end{equation}
Let us also recall from \eqref{eq definition of eta_C*} the associated involution
\[
\zeta_{f_0} : \LL \longrightarrow \LL.
\]

One can provide an explicit description of $\zeta_{f_0}$ in the particular case of symplectic surfaces.

\begin{proposition} \label{pr restriction of antisymplectic involutions give eta_KK}
Suppose that $S$ is a symplectic surface and $\zeta_S^+$ a symplectic ({\it resp.} $\zeta_S^-$ an antisymplectic) involution on it. Pick a smooth curve $C$ preserved by $\zeta_C$ and denote by $L$ its normal bundle inside $S$. Then $L = K_C$ and $\zeta_{f_0} = \zeta_\KK^+$, as defined in \eqref{eq eta_KK plus} ({\it resp.} $\zeta_{f_0} = \zeta_\KK^-$, as defined in \eqref{eq eta_KK minus}).
\end{proposition}

\begin{proof}
By adjunction formula and the triviality of the canonical bundle $K_S \cong \Oo_S$ of a symplectic surface, one gets $L = K_C$ in this case.

Note that the isomorphism $\partial \zeta_C : \zeta_C^*K_C \to K_C$ that lifts the action of $\zeta_C$ to the canonical bundle $K_C$, is the inverse of $d \zeta_C$.   

Since $\zeta_S^+$ is symplectic ({\it resp.} $\zeta_S^-$ is antisymplectic), $d \zeta_S : \Tt S \to \Tt S$ has eigenvalues whose product is $1$ ({\it resp.} $-1$). Then, over $C$, $d \zeta_S$ has eigenvalues $d \zeta_C$ and $f_0 = \partial \zeta_C$ ({\it resp.} $f_0 = -\partial \zeta_C$). This completes the proof after recalling \eqref{eq eta_KK plus} and \eqref{eq eta_KK minus}.
\end{proof}

We defined $\ol{\sS}$ in \eqref{eq definition of Ss} and  $\sS$ in \eqref{eq definition of Ss'}, having generic fibres $\ol{\sS}_t = \sS_t = S$ and central fibres $\ol{\sS}_0 = S \cup \LL$ and $\sS_0 = \Tot(L)$. As our involution $\zeta_S$ preserves $C$, $\zeta_S \times \id_{\PP^1}$ lifts to the blow-up $\ol{\sS}$ giving another involution $\zeta_{\ol{\sS}} : \ol{\sS} \to \ol{\sS}$ which preserves $\cC$ as defined in \eqref{eq definition of Cc}. Furthermore, at $t = 0$, $\zeta_{\ol{\sS}}$ preserves the irreducible component $S \subset \ol{\sS}_0$ and we denote by $\zeta_\sS : \sS \to \sS$ the restriction of this involution to $\sS$. One has commutativity with the structural morphisms,
\[
\xymatrix{
\sS \ar[rr]^{\zeta_\sS} \ar[rd] & & \sS \ar[ld]
\\
 & \PP^1, &
}
\]
giving $\zeta_S$ on every generic fibre at $t \neq 0$. We now examine the behaviour of $\zeta_\sS$ at the central fibre $\sS_0$.

\begin{lemma} \label{lm eta_sS at t=0}
With the notation as above, $\zeta_\sS$ coincides with $\zeta_S$ on the generic fibres at $t \neq 0$, and, over $t = 0$, with the restriction of $\zeta_{f_0}$ to $\Tot(L)$.
\end{lemma}

\begin{proof}
As $\ol{\sS}$ is defined by blowing-up $C \times \{ 0 \}$ inside $S \times \PP^1$, the involution given by $\zeta_{\ol{\sS}}$ is $\zeta_S$ on a generic fibre, and on the exceptional divisor $\LL$ is the involution induced by $d\zeta_S$ on $L$, {\it i.e.} $\zeta_{f_0}$. 
\end{proof}

Given $C \subset S$, we construct $\cC$ as in \eqref{eq definition of Cc} and we define $\{ n\cC \}_\sS$ as in \eqref{eq definition of Bb}. Consider a Mukai vector $\vV_a$ as in \eqref{eq definition of v}. Pick a polarization $\H$ on $S$, inducing $\hH$ as described in \eqref{eq definition of Hh}, and consider $\mM_{\sS/\AA^1, \hH}(\vV_a,n\cC)$ as defined in \eqref{eq definition of Mm}, classifying $\hH$-semistable relative sheaves on $\ol{\sS}$ with Mukai vector $\vV_a$, and whose support is parametrized by $\{n\cC\}_\sS$. In particular, their support restricts to $\sS$, so, taking pull-back under $\eta_\sS$, one obtains a birrational involution
\[
\birrat{\mM_{\sS/\AA^1, \hH}(\vV_a,n\cC)}{\mM_{\sS/\AA^1, \hH}(\vV_a,n\cC)}{\Ff}{\zeta_\sS^* \Ff,}{}{\widehat{\zeta}_\sS}
\]
compatible with the structural morphism to $\AA^1$.

Consider an $\zeta_S$-invariant line bundle $J \to S$ which satisfies condition \eqref{eq condition for J and v} with respect to $\v_a$. Starting from such $J$, we construct a line bundle on $\ol{\sS}$
\begin{equation} \label{eq definition of jJ}
\jJ = p_S^*J,
\end{equation}
where $p_S$ is obtained from the structural morphism of the blow-up $\ol{\sS} \to S \times \PP^1$ composed with the projection of $S \times \PP^1$ onto $S$. Note as well that, over $t \neq 0$, it restricts to the original line bundle,
\[
\jJ|_t \cong J,
\]
while, at $t = 0$, one gets 
\[
\jJ|_0 \cong p^*J_C,
\]
where $p$ is the structural projection $\Tot(L) \to C$, and
\[
J_C := J|_C.
\]
We then observe that \eqref{eq condition for J and v} holds fibrewise for $\jJ$ and $\vV_a$.

The following is inspired by \cite[Lemma 4]{sawon&shen}.

\begin{proposition-definition} \label{prop-def eta}
There exists a birregular involution
\begin{equation} \label{eq eEta}
\xi_{\jJ,\sS}: \mM_{\sS/\AA^1, \hH}(\vV_a,n\cC) \longrightarrow \mM_{\sS/\AA^1, \hH}(\vV_a,n\cC).
\end{equation}
given by sending a sheaf $\Ff$ supported on a curve inside $\sS|_t$ to $\Ee \! xt^2_{\ol{\sS}/\PP^1}(\Ff, \jJ(-\LL))|_\sS$. 

The restriction of $\xi_{\jJ,\sS}$ to a generic fibre over $t \neq 0$ is $\xi_{J,S}$, while at the central fibre on $t = 0$ it restricts to $\check{\xi}_{J_C,L}$.
\end{proposition-definition}

\begin{proof}
Recall that $\ol{\sS}$ is smooth with canonical bundle $K_{\ol{\sS}} = r^*K_{S \times \PP^1}(\LL)$ Consider the embedding $\jmath_t : \ol{\sS}_t \hookrightarrow \ol{\sS}$. The canonical bundle for this embedding is 
\[
\omega_{\jmath_t} = \Oo_{\ol{\sS}}(\LL)|_{\ol{\sS}_t},
\]
being trivial over a generic fibre over $t \neq 0$. Then, given a sheaf $\Ff$ supported on $\sS_t$ with $t \neq 0$, Grothendieck-Verdier duality implies that 
\[
\Ee \! xt^2_{\ol{\sS}/\PP^1}(\jmath_{t,*}\Ff, \jJ) \cong \Ee \! xt^1_{\ol{\sS}_t}(\Ff, \jJ|_t),
\]
so $\xi_{\jJ,\sS}$ restricts to $\xi_{J,S}$ there.

Also, observe that over for a sheaf supporte on the open subset of the central fibre $\Tot(L) = \sS_0 \subset \ol{\sS}_0$, one also has that  
\[
\Ee \! xt^2_{\ol{\sS}/\PP^1}(\jmath_{0,*}\Ff, \jJ|_0) \cong \Ee \! xt^1_{\ol{\sS}_0}(\Ff, \jJ|_0) \cong \Ee \! xt^1_{\LL}(\Ff, p^*J_C),
\]
so $\xi_{\jJ,\sS}$ restricts to $\xi_{J_0,\LL}$ which further corresponds to $\check{\xi}_{J_C,L}$ after Lemma \ref{lm involutions and spectral correspondence}.

Finally, since $\xi_{J,S}$ and $\check{\xi}_{J_C,L}$ preserves stability, so does $\xi_{\jJ,\sS}$ fibrewise, defining a biregular involution of the moduli space. 
\end{proof}

Observe that $\jJ$ is $\zeta_\sS$-invariant. Then, $\widehat{\zeta}_\sS$ and $\xi_{\jJ,\sS}$ commute and one can consither the composition 
\[
\lambda_{\jJ,\sS} := \widehat{\zeta}_\sS \circ \xi_{\jJ,\sS} : \mM_{\sS/\AA^1, \hH}(\vV_a,n\cC) \dashrightarrow \mM_{\sS/\AA^1, \hH}(\vV_a,n\cC).
\]

\begin{lemma} \label{lm fibrewise description of lambda}
With the notation as above, $\lambda_{\jJ,\sS}$ coincides with $\check{\lambda}_{J_C,f_0}$ over $t = 0$, and with $\lambda_{J,S}$ over a generic fibre at $t \neq 0$. 
\end{lemma}

\begin{proof}
The lemma is a consequence of Proposition-definition \ref{prop-def eta} and the fibrewise description of $\jJ$ and $\vV_a$.
\end{proof}

We now study the closure of their fixed points and their structure morphism.

\begin{proposition} \label{pr flatness}
The restriction of the structural morphism to $\AA^1$ of the closed subvarieties $\overline{\Fix(\widehat{\zeta}_\sS)}$, $\Fix(\xi_{\jJ,\sS})$ and $\overline{\Fix(\lambda_{\jJ,\sS})}$ is flat.
\end{proposition}

\begin{proof}
Since $\mM_{\sS/\AA^1, \hH}(\vV_a,n\cC)$ is flat over $\AA^1$, which is an irreducible and one dimensional base scheme, and the structural morphism trivializes (hence, it is flat) over $\AA^1 -\{ 0 \}$, one only needs to check that the fibre at $t = 0$ arises as the closure of $\overline{\Fix(\widehat{\zeta}_\sS)}|_{\AA^1 -\{ 0 \}}$, $\Fix(\xi_{\jJ,\sS})|_{\AA^1 -\{ 0 \}}$ and $\overline{\Fix(\lambda_{\jJ,\sS})}|_{\AA^1 -\{ 0 \}}$ inside $\mM_{\sS/\AA^1, \hH}(\vV_a,n\cC)$. Note that this follows trivially for $\xi_{\jJ,\sS}$ as it is birregular, but also for $\widehat{\zeta}_\sS$ and $\lambda_{\jJ,\sS}$, which are biregular at the central fibre.
\end{proof}

Suppose now that $S$ denotes a smooth surface equipped with involution $\zeta_S : S \to S$ whose quotient map $\q_S : S \to T := S/\zeta_S$ gives a smooth surface. Pick a smooth curve $D \subset T$ and consider $C$ as in \eqref{eq C from D}, which is again smooth. Denote by $\zeta_C: C \to C$ the restriction of $\zeta_S$ to $C$, and by $\q_C : C \to D=C/\zeta_C$ the restriction there of $\q_S$. As in \eqref{eq setting L}, denote the associated normal bundles by $L= \Oo_S(C)|_C$, as well as $W=\Oo_T(D)|_D$. Consider as well the ruled surfaces $\LL$ and $\WW$ associated to $L$ and $W$, as defined in \eqref{eq definition of LL}. Note that the composition with $d \zeta_S : \zeta_S^* \Tt S \to \Tt S$ provides an isomorphism $f_0 : \zeta_C^*L \stackrel{\cong}{\to} L$ and recall from \eqref{eq definition of eta_C*} the associated involution $\zeta_{f_0}$, as well as the projection $\q_{f_0}$ from \eqref{eq definition of q_LL}.

\begin{lemma} \label{lm eta_Ss_0 = q_f_0}
Lift, via $f_0$ as above, the action of $\zeta_C$ to the normal line bundle $L$ of $C$ inside $S$. Then, $L$ descends under this action to the normal bundle $W$ of $D$ inside $T$. Hence,
\[
\WW \cong \quotient{\LL}{\zeta_{f_0}}.
\]
\end{lemma}

\begin{proof}
The first statement follows easily as the tangent bundles $\Tt\, T$ and $\Tt D$ coincide, respectively, with the $\zeta_S$-invariant and $\zeta_C$-invariant subbundles of $\Tt S$ and $\Tt C$. The second statement follows from the first and Remark \ref{rm descent}.
\end{proof}

Following \eqref{eq definition of Ss} and \eqref{eq definition of Ss'}, define $\ol{\tT}$ as well as $\tT$, having generic fibres $\ol{\tT}_t = \tT_t = T$, over $t \neq 0$, and central fibre $\ol{\tT} = T \cup \WW$ and $\tT_0 = \Tot(W)$. Since the preimage of $D \times \{ 0 \}$ under the projection $\q_S \times \id_{\PP^1} : S \times \PP^1 \to T \times \PP^1$ is the divisor $\LL \subset \ol{\sS}$, one obtains the projection $\q_{\ol{\sS}} : \ol{\sS} \to \ol{\tT}$ commuting with the structural morphisms to $\PP^1$. At $t = 0$, $\q_{\ol{\sS}}$ sends the irreducible component $S$ of $\ol{\sS}_0$ to the irreducible component $T$ of $\ol{\tT}_0$. Then, by restricting $\q_{\ol{\sS}}$ to $\sS$, we can define $q_{\sS} : \sS \to \tT$ commuting with the structural morphisms,  
\begin{equation} \label{eq sS to tT}
\xymatrix{
\sS \ar[rr]^{\q_\sS} \ar[rd] & & \tT \ar[ld]
\\
 & \PP^1. &
}
\end{equation}

\begin{lemma} \label{lm q_sS at t=0}
With the notation as above, $\q_\sS$ coincide with $\q_S$ on the generic fibres at $t \neq 0$, and, with the restriction of $\q_{f_0}$ to $\Tot(L)$ over $t = 0$. Hence, one has that $\q_\sS$ is the quotient map of the action of $\zeta_\sS$.
\end{lemma}

\begin{proof}
After Lemmas \ref{lm eta_sS at t=0} and \ref{lm eta_Ss_0 = q_f_0}, one has that the restriction of $\q_{\ol{\sS}}$ to $\LL$ coincides with $\q_{f_0}$, which is the quotient map associated to $\zeta_{f_0}$. This concludes the proof, as we knew already that $\q_\sS$ is the quotient map of $\zeta_\sS$ over the complement of the central fibre.  
\end{proof}

We construct $\dD$ as in \eqref{eq definition of Cc} starting from the smooth curve $D \subset T$. As defined in \eqref{eq definition of Bb}, denote the connected components of the corresponding relative Hilbert schemes containing the family of curves $n\dD$, once we restrict to $\tT$. Denote also $\q_\sS^*\{ n\dD \}_\tT \subset \{ n\cC \}_\sS$ the locus of curves lifted from $T$. S


Let us consider the restriction to $\q_\sS^{-1}\{ n\dD \}_\tT$ of the closure of the locus fixed by $\widehat{\zeta}_\sS$,
\[
\nN_{\sS/\AA^1, \hH}(\vV_a,n\dD) := \overline{\Fix(\widehat{\zeta}_\sS)} \cap \mathbf{h}^{-1}(\q_\sS^{-1}\{ n\dD \}_\tT).
\]
This provides a degeneration of the fixed locus of a natural involution on the moduli space of sheaves over $S$.

\begin{theorem} \label{tm degeneration of N}
Let $S$ be a smooth projective surface equipped with an involution $\zeta_S$ whose quotient $T$ is smooth, pick a smooth curve $D \subset T$ and consider its lift $C$ to $S$. Pick as well a positive integer $n$ and, associated to $n$ and $C$, a Mukai vector $\v_a$ as in \eqref{eq definition of v_a}. Then, there exists a closed subvariety $\nN_{\sS/\AA^1, \hH}(\vV_a,n\dD) \subset \mM_{\ol{\sS}/\AA^1, \hH}(\vV_a,n\cC)$, flat over $\AA^1$, and equipped with a surjective support morphism  
\begin{equation} \label{eq commuting diagram Nn_Y}
\xymatrix{
\nN_{\sS/\AA^1, \hH}(\vV_a,n\dD) \ar[rr] \ar[rd] & & \q^*_\sS \{ n\dD \}_\tT \ar[ld]
\\
& \AA^1, &
}
\end{equation}
such that the structural morphism is trivial over $\AA^1 - \{ 0 \}$, the generic fibre over $t \neq 0$ is the closure of the natural involution on $\M_{S,\H}(\v_a)$ associated to $\zeta_S$ equipped with the fibration \eqref{eq support morphism for N}, 
\[
\N_{S,\H}(\v_a,nD) \to \q_S^*\{nD\},
\]
while the central fibre at $t = 0$ is the fixed locus of the involution $\check{\zeta}_{f_0}$ given in \eqref{eq description of hat eta}, being $f_0$ induced from $d\zeta_S$ as in \eqref{eq definition of f_0}, endowed with the Hitchin morphism \eqref{eq Hitchin fibration for N},
\[
\Nn_C^L(n,d) \to \bigoplus_{i = 1}^n \q_C^*H^0(D, W^{\otimes i})
\]
where $d = a - \delta$.
\end{theorem}

\begin{proof}
After Lemma \ref{lm eta_sS at t=0}, the fibrewise description is a straight-forward since the restriction of $\widehat{\zeta}_\sS$ to a generic fibre over $t \neq 0$ coincides with $\zeta_S$, while, at the central fibre over $t = 0$, $\widehat{\zeta}_\sS$ restricts to $\check{\zeta}_{f_0}$. Then, the description of the generic fibres is provided by the definitions \eqref{eq N over B} and \eqref{eq support morphism for N}, while the description of the central fibre follows from Lemmas \ref{lm Nn supported on vector space} and \ref{lm eta_Ss_0 = q_f_0}. Finally, recall Proposition \ref{pr flatness} for the proof of flatness.
\end{proof}

For $T$, consider $\mM_{\tT/\AA^1, \iI}(\wW_b, n\dD)$ of sheaves with a relative Mukai vector $\wW_b$ and a polarization $\I$ inducing $\iI$ and $\widetilde{\iI} = \q_\sS^* \iI$. In this case, $\widehat{\zeta}_\sS$ is biregular and the image of the pull-back morphism under the corresponding quotient map lies in its fixed locus,
\begin{equation} \label{eq definition of hatqQ}
\morph{\mM_{\tT/\AA^1. \iI}(\wW_b, n\dD)}{\nN_{\sS/\AA^1, \widetilde{\iI}}(\vV_{2b}, n\dD) \subset \mM_{\sS/\AA^1, \widetilde{\iI}}(\vV_{2b}, n\cC)}{\Ff}{\q_\sS^*\Ff.}{}{\widehat{\q}_\sS}
\end{equation}
Note that over the generic fibre, $t \neq 0$, \eqref{eq definition of hatqQ} restricts to \eqref{eq definition of hatq}, while on the central fibre, $t = 0$, the restriction of \eqref{eq definition of hatqQ} is \eqref{eq definition of hatq HB}. Both cases are covered by Proposition \ref{pr q embedding on ssl} which provides the following corollary.

\begin{corollary} \label{co qQ on ssl}
Hence, \eqref{eq definition of hatqQ} is generically 1:1 when the ramification locus of $\q_S$ intersects $C$ non-trivially, and, otherwise (in particular when $\zeta_S$ is unramified), generically $2:1$.
\end{corollary}

As we did in the case of a natural involution, consider
\[
\pP_{\sS/\AA^1, \hH}(\vV_a,n\dD) := \overline{\Fix(\lambda_{\jJ,\sS})} \cap \mathbf{h}^{-1}(\q_\sS^{-1}\{ n\dD \}_\tT),
\]
which gives a degeneration of the fixed locus of $\lambda_{J, S}$.

\begin{theorem} \label{tm degeneration of P}
Let $S$ be a smooth projective surface equipped with an involution $\zeta_S$ whose quotient $T$ is smooth, pick a smooth curve $D \subset T$ and consider its lift $C$ to $S$. Choose a positive integer $n$ and take a Mukai vector $\v_a$ as in \eqref{eq definition of v_a}
and a line bundle $J$ on $S$ satisfying \eqref{eq condition for J and v}. 

Then, there exists a closed subvariety  $\pP_{\sS/\AA^1, \hH}(\vV_a, \jJ, n\dD) \subset \mM_{\sS,\hH}(\vV_a,n\cC)$, flat over $\AA^1$, and equipped with a surjective support morphism  
\begin{equation} \label{eq commuting diagram Pp_Y}
\xymatrix{
\pP_{\sS/\AA^1, \hH}(\vV_a, \jJ,n\dD) \ar[rr] \ar[rd] & & \q^*_\sS \{ n\dD \}_\tT \ar[ld]
\\
& \AA^1, &
}
\end{equation}
such that the structural morphism is trivial over $\AA^1 - \{ 0 \}$, the generic fibre over $t \neq 0$ is the closure of the involution $\lambda_{J,S}$ on $\M_{S,\H}(\v_a)$ equipped with the fibration \eqref{eq support morphism for P}, 
\[
\P_{S,\H}(\v_a, J,nD) \to q_S^*\{ nD \},
\]
while the central fibre at $t = 0$ is the fixed locus of the involution $\check{\lambda}_{J_C, f_0}$ given in \eqref{eq lambda}, being $f_0$ induced from $d\zeta_S$ as in \eqref{eq definition of f_0} and $J_C$ the restriction of $J$ to $C$, endowed with the Hitchin morphism \eqref{eq Hitchin fibration for P},
\[
\Pp^L_C(n,d, J_C) \to \bigoplus_{i = 1}^n \q_C^*H^0(D, W^{\otimes i}).
\]
where $d = a - \delta$.
\end{theorem}

\begin{proof}
The proof is analogous to the proof of Theorem \ref{tm degeneration of N} and follows naturally from Lemmas \ref{lm Pp supported on vector space}, \ref{lm eta_sS at t=0} and \ref{lm eta_Ss_0 = q_f_0} and Propositions \ref{prop-def eta} and \ref{pr flatness}.
\end{proof}

\subsection{Degenerations of natural Lagrangians and Prymian integrable systems}

\label{sc degeneration of ASF systems}

Sawon conjectured in \cite{sawon} that the Prymian integrable systems constructed by Markushevich--Tikho\-mi\-rov \cite{markushevich&tikhomirov}, Arbarello--Sacc\`a--Ferretti \cite{ASF} and Matteini\cite{matteini}, degenerate into integrable systems related to the Hitchin system, leaving open the description of these conjectural systems· 

In this section we provide an explicit description of the degenerations of Section \ref{sc degenerating involutions} in the cases studied in Section \ref{sc Prymian integrable systems}, focusing on the cases specified in \cite{markushevich&tikhomirov,ASF,matteini, sawon&shen2, shen}, which are constructed using primitive first Chern classes. We also review the degeneration given by Sawon--Shen \cite{sawon&shen} into the $\Sp(2m,\CC)$-Higgs moduli space, studying as well the degeneration of the associated natural Lagrangian. It is worth noticing that we find that this natural Lagrangian degenerates into the fixed locus of an involution associated to $\U(m,m)$-Higgs bundles, whose Nadler--Langlads group is $\Sp(2m,\CC)$.

Our work provides an answer to the question posed by Sawon in \cite{sawon} (which, strictly speaking, refers only to the case of primitive first Chern classes) and to the reformulation of Sawon's question on the non-primitve case.

\subsubsection{The general case}

Consider a K3 surface $X$ equipped with an antisymplectic involution $\zeta_X^- : X \to X$, whose quotient $\q_X : X \to Y := X /\zeta_X^-$ has ramification divisor $\Delta_X$. Let $D \subset Y$ be a smooth curve and $C$ as in \eqref{eq C from D}, again smooth. We denoted by $\zeta_C: C \to C$ and $\q_C : C \to D=C/\zeta_C$ the restriction of $\zeta_X^-$ and $\q_X$ to $C$. Note that the ramification divisor of $\zeta_C$ is $\Delta_C = C \cap \Delta_X$. Since $K_X$ is trivial, the normal bundle of $C$ is
\[
L = K_C,
\]
giving the ruled surface $\KK_C$, and the normal bundle of $D$ is 
\begin{equation} \label{eq description of W}
W = K_D K_Y^{-1}|_D,
\end{equation}
associated to the ruled surface $\WW$. Observe that $K_Y^{-1}|_D$ is a square-root of the branching divisor of $\q_C$, and its pull-back is the ramification divisor.

Recall from Proposition \ref{pr restriction of antisymplectic involutions give eta_KK}, that $d \zeta_X^- : (\zeta_X^-)^* \Tt X \to \Tt X$ gives the isomorphism $-\partial_C: \zeta_C^*K_C \stackrel{\cong}{\to} K_C$ giving the associated involution $\zeta_\KK$ described in \eqref{eq eta_KK minus}. The associated natural antisymplectic involution $\widehat{\zeta}_\KK^-$ is sent to $\check{\zeta}_K^-$ under the spectral correspondence as described in \eqref{eq eta K}. As a consequence of Proposition \ref{pr restriction of antisymplectic involutions give eta_KK}, Theorems \ref{tm relative symplectic form} and \ref{tm degeneration of N} and Corollary \ref{co qQ on ssl} one has the following.

\begin{corollary} \label{co degeneration of N^-}
There exists a non-linear degeneration $\nN_{\xX,\hH}(\vV_a, n\dD)$ of the natural Lagrangian subvariety $\N^{\H}_X(\v_a, nD) \subset \M_{X,\H}(\v_a, nC)$ obtained from the fixed locus of $\zeta_X^-$, into the Lagangian subvariety of the Higgs moduli space $\Nn_C(n, d) \subset \Mm_C(n, d)$, with $d = a - n(n-1)(g_C - 1)$, given by the fixed locus of the involution 
\begin{equation} \label{eq zeta K minus}
\morph{\Mm_C(n,d)}{\Mm_C(n,d)}{(E,\varphi)}{(\zeta_C^*E, -\zeta_C^*\varphi).}{}{\check{\zeta}_K^-}
\end{equation}

For $a = 2b$ and a $\zeta_X$-invariant polarization $\widetilde{\I} = \q_X^*\I$, the morphism 
\[
\widehat{\q}_\xX : \mM_{\yY/\AA^1}^\iI(\wW_b, n\dD) \longrightarrow   \nN_{\xX/\AA^1}^{\widetilde{\iI}}(\vV_{2b}, n\dD) \subset \mM_{\xX/\AA^1}^{\widetilde{\iI}}(\vV_{2b}, n\cC),
\]
given in \eqref{eq definition of hatqQ}, is generically 1:1 when the ramification locus of $\q_S$ intersects $C$ non-trivially, and generically $2:1$ otherwise, for instance whenever $\zeta_S$ is unramified. The generic fibres of $\mM_{\yY/\AA^1}^\iI(\wW_b, n\dD)$ are the moduli spaces $\M_{Y,\I}(\w_b, nD)$ and the central fibre is $\Mm_D^{W}(n,d')$, for $W$ described in \eqref{eq description of W} and $d' = b - n(n-1)(g_D - 1)$.
\end{corollary}


As a consequence of Proposition \ref{pr restriction of antisymplectic involutions give eta_KK} and Theorems \ref{tm relative symplectic form} and \ref{tm degeneration of P} one has the following.

\begin{corollary} \label{co degeneration for P^-}
Suppose that \eqref{eq condition for J and v} holds for the choice of a line bundle $J$ and a Mukai vector $\v_a$ as in \eqref{eq definition of v_a} for some integer $a$ and let $d$ be $a - \delta= a - n(n-1)(g_C - 1)$. 

There exists a non-linear degeneration $\pP_{\xX,\hH}(\vV_a, \jJ, n\dD)$ of the Prymian integrable system $\P_{X,\H}(\v_a, J, nD) \subset \M_{X,\H}(\v_a, nC)$ constructed out of $\zeta_X^-$, into the Prymian integrable system $\Pp_C(n, d, J_C) \subset \Mm_C(n, d)$ given by the fixed locus of the involution 
\begin{equation} \label{eq lambda K minus}
\morph{\Mm_C(n,d)}{\Mm_C(n,d)}{(E,\varphi)}{(\zeta_C^*E^*, -\zeta_C^*\varphi).}{}{\check{\lambda}_{J_C,K}^-}
\end{equation}
\end{corollary}

\begin{remark}
For $n=1$, $\Pp_C(n, d, J_C)$ is the cotangent bundle of a Prym abelian variety associated to $\zeta_C$ (when $\zeta_X^-$ is ramified) or a disjoint union of them (when $\zeta_X^-$ unramified). 
\end{remark}

\subsubsection{The Arbarello--Sacc\`a--Ferretti systems and Sacc\`a Calabi-Yau's}

\label{sc ASF}

In this subsection, we consider a K3 surface endowed with an antisymplectic involution $\zeta_X^-$ with empty fixed locus and whose quotient $Y = X/\zeta_X$ is a smooth Enriques surface. The canonical divisor of an Enriques surface $Y$ is not trivial but satisfies $2 K_Y \sim 0$. Any smooth curve $D \subset Y$ with positive self-intersection is big and nef, and so is $D + K_Y$. Then, by Kodaira vanishing theorem, 
\[
h^1(Y,D) = h^2(Y,D) = 0,
\]
hence
\[
h^0(Y,D) = 1 + \frac{1}{2} D^2 = g_D,
\]
where $g_D$ is the genus of $D$. The curve $C := D \times_Y X$ is equipped with an unramified $2:1$ cover $\q_C \to D$ and an involution $\zeta_C : C \to C$ without fixed points.

Observe that condition \eqref{eq Lagrangian condition} is satisfied for every smooth curve $D \subset Y$. Then, whenever \eqref{eq condition for J and v} holds for a pair $\J$ and $\v_a$, $\P_{X,\H}(\v_a, \J, nD)$ has $2$ connected components each of them an integrable system of dimension $2n^2(g_D - 1)$ known as the {\it Arbarello--Sacc\`a--Ferretti} integrable system, as described by these authors in \cite{ASF} when $n=1$. The Lagrangian subvariety $\N_{X,\H}(\v_a, nD)$ has dimension $n^2(g_C -1) +1$. Furthermore, Sacc\`a \cite{sacca} studied the moduli space of pure dimension $1$ sheaves over an Enriques surface $\M_{Y,I}(\w_b, nD)$ showing that it is Calabi--Yau. By Proposition \ref{pr q embedding on ssl}, $\M_{Y,I}(\w_b, nD)/\ZZ_2$ is birrational to $\N_{Y,\widetilde{I}}(\v_{2b}, nD)$. 

Observe that $L \cong K_C$ in this case, and $W \cong K_D L_\gamma$, where $L_\gamma \in H^1(D, \ZZ_2)$ is the $2$-torsion line bundle obtained by restricting $K_Y$ to $D$. After Corollary \ref{co degeneration of N^-} $\M_{Y,I}(\w_b, nD)$ degenerates into $\Mm^{L_\gamma K_D}_D(n,b-\delta/2)$ and $\N_{X,\H}(\v_a, nD)$ into $\Nn_C(n,a - \delta)$, which is the fixed point locus of the involution \eqref{eq zeta K minus} constructed out of $\zeta_C$ unramified. It follows from Corollary \ref{co degeneration for P^-} that $\P_{X,\H}(\v_a, J, nD)$ degenerates into $\Pp_C(n,a - \delta, J_C)$, which is the fixed point locus of the involution \eqref{eq lambda K minus}, constructed out of $\zeta_C$ unramified.

\subsubsection{Markusevich--Tickhomirov, Matteini and Sawon--Shen systems}

\label{sc MT}

Consider now an anti-symplectic involution $\zeta_X^-$ on the K3 surface $X$ whose quotient $Y = X/\zeta_X^-$ is a del Pezzo surface of degree $d$. In that case, by their defining property, the anti-canonical bundle $-K_Y$ is ample. In particular, $-(m+1)K_Y$ is big and nef for $m \geq 0$ and by Kodaira vanishing theorem, 
\[
h^1(Y,-mK_Y) = h^2(Y,-mK_Y) = 0,
\]
so
\[
h^0(Y,-mK_Y) = 1 + \frac{m^2 + m}{2} \cdot d.
\]
Note that the genus of a smooth curve in the linear system $-mK_Y$ is $1 + (m^2 - m)d/2$. 

We recall that $\Delta$, the branching locus of $\q_X$, lies in $|-2K_Y|$. Picking a smooth curve $D$ in the linear system $|-K_Y|$, hence $g_D = 1$, pulling-back to another smooth curve $C$ on $X$, we obtain a $2:1$ cover, $\q_C : C \to D$, ramified at the intersection with the ramification locus of $\q_X$, that we denote by $R$. Observe that the length of $R$ coincides with the intersection of $D$ and $\Delta$, so $(-K_Y) \cdot (-2K_Y) = 2d$, form where one can derive that $g_C = 1 + d$. 
 
Markushevich and Tikhomirov \cite{markushevich&tikhomirov} gave the first construction of Prymian intagrable system $\P_{X,\H}(\v_a, J, nC)$ starting from a del Pezzo of degree $d=2$ and for $n=1$. In this case we have that $\zeta_C$ sends a genus $3$ curve $C$, to $D$, an elliptic curve and $\zeta_C$ is ramified at $4$ points and $\P^{\H}_X(\v_a, J, C)$ is a $4$-dimensional symplectic $V$-manifold. 

Matteini extended the construction of \cite{markushevich&tikhomirov}, focusing on the cases of del Pezzo surfaces of degree $d=1$ and $d = 3$. In the case of $d=1$, $\zeta_C : C \to D$ is the cover of a genus $2$ curve over an elliptic curve, and Matteini found that $\P_{X,\H}(\v_a, J, nC)$ is an elliptic K3 surface. In the $d=3$ case, $\zeta_C : C \to D$ is a $2:1$ cover of a genus $4$ curve onto an elliptic curve and $\P_{X,\H}(\v_a, J, C)$ is a $6$-dimensional symplectic $V$-manifold. Even in both cases only $n=1$ was considered by the authors, their construction works for general $n$ as indicated by Matteini in \cite[Section 3.6]{matteini}.

Another $6$-dimensional Prym integrable system is considered by Sawon and Shen in \cite{sawon&shen2, shen}. In this case $Y$ is a del Pezzo surface of degree $d=1$ and $D \subset Y$ lies in $|-2K_Y|$, so $g_D = 2$. The ramification divisor $R$ of the $2:1$ cover $\q_C : C \to D$ has length $(-2K_Y) \cdot (-K_Y) = 2$, hence $g_C = 5$.  

After Corollary \ref{co degeneration for P^-} that $\P_{X,\H}(\v_a, J, nD)$ degenerates into $\Pp_C(n,a - \delta, J_C)$, which is the fixed point locus of the involution \eqref{eq lambda K minus}, constructed out of $\zeta_C$ ramified at $\Delta_C = \Delta_X \cap X$.

One can also consider the natural Lagrangian subvarieties $\N_{X,\H}(\v_a, nD)$ associated to the mentioned Prymian integrable systems. Note that in this case $\N_{X,\widetilde{\I}}(\v_{2b}, nD)$ are birrational to $\M_{Y,\I}(\v_{b}, nD)$. From Corolloary \ref{co degeneration of N^-} we know that $\N_{X,\H}(\v_a, nD)$ degenerates into $\Nn_C(n,a - \delta)$, which is the fixed point locus of the involution \eqref{eq zeta K minus}, and $\M_{Y,\I}(\v_{b}, nD)$ degenerates into $\Mm_D^{K_DQ}(n,b-(n^2-n)(g_D - 1))$, where $Q$ is a line bundle on $D$ whose square is $\Oo_D(\Delta_C)$.

\subsubsection{Sawon--Shen degeneration into the $\Sp(2m,\CC)$-Hitchin system}

\label{sc sawon and shen}

When $\zeta_X^-$ is an anti-sym\-plec\-tic involution on the K3 surface $X$ giving a del Pezzo surface $Y = X / \zeta_X^-$ of degree $d$, the ramification divisor of $\q_X : X \to Y$ is a connected smooth projective curve $\Delta$ of genus $g_\Delta = d + 1$. Trivially, the restriction of $\zeta_X$ to $\Delta$ is the identity $\zeta_\Delta = \id_\Delta$. Setting $\v_\delta = (0, 2m\Delta, 1- (4m^2 - 2m)(g_\delta - 1))$, $\H = \Oo_X(\ell\Delta)$ and $J = \Oo_X(-2m\Delta)$, Sawon and Shen \cite{sawon&shen} have constructed a degeneration of the Prymian integrable system $\P_{X,\H}(\v_\delta, J, 2m\Delta)$ into a compactification of the $\Sp(2m,\CC)$-Hitchin system over $\Delta$, given by the fixed locus of the involution 
\[
\morph{\Mm_\Delta(n,0)}{\Mm_\Delta(n,0)}{(E,\varphi)}{(E^*, -\varphi^t).}{}{\check{\lambda}_{K^{-2m}_\Delta,K}^-}
\]

Note that Corollary \ref{co degeneration of N^-} shows that the natural Lagrangian $\N_{X,\H}(\v_\delta)$ degenerates into the fixed locus of the involution
\[
\morph{\Mm_\Delta(n,0)}{\Mm_\Delta(n,0)}{(E,\varphi)}{(E, -\varphi).}{}{\check{\zeta}_{K}^-}
\]
The morphism obtained by extension of structure group sends the moduli space of $\U(m,m)$-Higgs bundles maps to the fixed locus of $\check{\zeta}_{K}^-$. It is worth noticing that $\U(m,m)$ and $\Sp(2m,\CC)$ are related under the Nadler-Langlads correspondence \cite{nadler}. In \cite{hitchin_char} Hitchin proposed that a $\BBB$-brane constructed over the moduli space of $\Sp(2m,\CC)$-Higgs bundles is Mirror dual to a $\BAA$-brane supported on the moduli space of $\U(m,m)$-Higgs bundles (see Section \ref{sc branes}). Evidence for this duality was given by Hitchin himself in \cite{hitchin_char} and by Hausel, Mellit and Pei in \cite{hausel&mellit&pei}.


\section{Branes and duality}

\label{sc branes}

Non-abelian Hogde theory \cite{corlette, donaldson, hitchin-self, simpson1, simpson2} implies that $\Mm_C(n,d)$ is equipped with a hyperk\"ahler structure $(g, \Gamma_1, \Gamma_2, \Gamma_3)$, where $\Gamma_1$ the complex structure coming from the moduli space of Higgs bundles and $\Gamma_2$ from the moduli of flat connections. The associated Kahler forms $\omega_i(\, \cdot \, , \, \cdot \, ) := g(\, \cdot \, , \Gamma_i (\, \cdot \,) )$ combine into a $\Gamma_1$-holomorphic symplectic form $\omega_2 + \mathrm{i} \omega_3$, which coincides with $\Omega_0$ up to scaling. Similar constructions can be given for $i=2$ and $i=3$.

Following \cite{kapustin&witten}, a {\it $\BAA$-brane} in $\Mm_C(n,d)$ is a pair $(N, \Ggg, \nabla_\Ggg)$ where $N$ is a subvariety which is complex Lagrangian with respect to $\Omega_0$, and $(W, \nabla_W)$ is a flat bundle over $N$. A {\it $\BBB$-brane} in $\Mm_C(n,d)$ is given by a pair $(P, \Fff, \nabla_\Fff)$, where $P$ is a hyperholomorphic subvariety({\it i.e.} holomorphic with respect to $\Gamma_1$, $\Gamma_2$ and $\Gamma_3$) and $(\Fff, \nabla_\Fff)$ a hyperholomorphic sheaf on $P$. This means that the connection $\nabla_{\Fff}$ on the sheaf $\Fff$ is of type $(1,1)$ with respect to all three K\"ahler structures $\Gamma_1$, $\Gamma_2$, $\Gamma_3$. It is conjectured in \cite{kapustin&witten} that mirror symmetry interchanges $\BBB$-branes with $\BAA$-branes within the Higgs moduli space and in a certain limit, mirror symmetry is enhanced via a Fourier-Mukai transform relative to the Hitchin fibration.

Note that the definition of $\BAA$ and $\BBB$-branes extends naturally to any other hyperk\"ahler varieties. The case of the moduli space $\M_{S,\H}(\v_a)$ of sheaves over a smooth projective symplectic surface $S$ was considered in \cite{GYM} by the author, Jardim and Menet.   It is also described in \cite{GYM} the construction of the  $\BBB$-brane $\N_{S,\H}^+(\v_a)$ arising from the fixed locus of a natural involution associated to a sympletic involution on the surface. Also, the authors studied the behaviour of these natural branes under some correspondences. As a consequence of Propositions \ref{pr restriction of antisymplectic involutions give eta_KK} and \ref{pr flatness}, we can describe these degenerations in the context of branes, extending the description to $\P_{S,\H}(\v_a, J)$.

\begin{corollary} \label{co degeneration of branes}
The $\BBB$-brane $\N_{S,\H}^+(\v_a)$ ({\it resp.} the $\BAA$-brane $\N_{S,\H}^-(\v_a)$) arising from the fixed locus of a natural involution associated to a sympletic ({\it resp.} anti-symplectic) involution on a symplectic surface degenerate into the $\BBB$-brane ({\it resp.} $\BAA$-brane) inside the Hitchin system $\Nn_C^+(n,d)$ ({\it resp.} $\Nn_C^-(n,d)$), obtained as the respective fixed locus of the involution $\check{\zeta}_K^+$ ({\it resp.} $\check{\zeta}_K^-$), described in \eqref{eq eta K}.

Similarly, the $\BBB$-brane obtained from $\P_{S,\H}(\v_a, J)$ degenerates into the $\BBB$-brane given by $\Pp_C(n,d, J_C)$. 
\end{corollary}

We now discuss the duality under Mirror symmetry of the branes associated to a Prymian integrable sysem and the natural Lagrangian.

Consider the Mukai vector $\v_0$ on $S$, where we assume $a = 0$, and recall from Remark \ref{rm v_a eta-invariant} that is is the pull-back of the Mukai vector $\w_0$ on $T$. Following Remark \ref{rm choice of J_0 and v_0}, choose $J_0 = \Oo_S(-nC)$. In this section we show that certain open subsets of the subvarieties $\N_{S,\H}(\v_0,B)$ and $\P_{S,\H}(\v_0,J_0,B)$ are dual under Fourier--Mukai transform.

Observe that the diagram,
\[
\xymatrix{
\M_{S,\H}(\v_0) \ar[rrdd]^{\h_S} & \ar@{_(->}[l] \P_{S,\H}(\v_0, J_0,B) \ar[rd]^{\nu} &  & \N_{S,\H}(\v_0,B) \ar@{^(->}[r] \ar[dl]_{\mu} & \M_{S,\H}(\v_0) \ar[lldd]_{\h_S}
\\
& & \q_S^*\{ B \} \ar@{^(->}[d] & &
\\
& & \{ A \}, & &
}
\]
restricts, after Proposition \ref{pr q embedding on ssl} and Remark \ref{rm Prym at ssl}, to 
\[
\xymatrix{
\Jac_{\Aa^\sm / \{ A \}^\sm} \ar[rrdd]^{\h_S} & \ar@{_(->}[l] \Prym \left (\q_{\Aa^\ssl}/{\{ B \}^{\ssl}}\right ) \ar[rd]^{\nu} &  & \Jac_{\Bb^{\, \ssl} / \{ B \}^\ssl} \ar@{^(->}[r]^{\widehat{\q}_{\Aa^\ssl}} \ar[dl]_{\mu} & \Jac_{\Aa^\sm / \{ A \}^\sm} \ar[lldd]_{\h_S}
\\
& & \q_S^*\{ B \}^\ssl \ar@{^(->}[d] & &
\\
& & \{ A \}^\sm, & &
}
\]
where we have dropped the systematic reference to degree $0$ in our notation for Jacobians.

Let us recall that one can equip $\Jac_{\Aa^\sm/\{A \}^\sm} \times_{\{ A \}^\sm} \Jac_{\Aa^\sm/\{A \}^\sm}$ with a relative Poincar\'e bundle $\Uu_{\Aa^\sm}$. Denoting by $\pi_i$ the projection from $\Jac_{\Aa^\sm/\{A \}^\sm} \times_{\{ A \}^\sm} \Jac_{\Aa^\sm/\{A \}^\sm}$ to the $i$-th factor, one can consider the relative Fourier--Mukai transform \cite{mukai_fourier}
\[
\morph{\Dd^b(\Jac_{\Aa^\sm/\{A \}^\sm})}{\Dd^b(\Jac_{\Aa^\sm/\{A \}^\sm})}{\Ff}{R\pi_{2,*} \left ( \pi_1^* \Ff \otimes \Uu_{\Aa^\sm} \right ),}{}{\Psi_{\Aa^\sm}}
\]
and its inverse,
\[
\morph{\Dd^b(\Jac_{\Aa^\sm/\{A \}^\sm})}{\Dd^b(\Jac_{\Aa^\sm/\{A \}^\sm})}{\Gg}{R\pi_{1,*} \left ( \pi_2^* \Gg \otimes \Uu^{-1}_{\Aa^\sm} \right ).}{}{\Psi_{\Aa^\sm}^{-1}}
\]
After the choice of $J_0$ and \eqref{eq eta_J C'} and the properties of Fourier--Mukai transforms (see \cite{polishchuck} or instance), the composition is the auto-equivalence of derived categories
\[
\Psi_{\Aa^\sm} \circ \Psi_{\Aa^\sm}^{-1} = \xi_{J_0, S} [-g_A].
\]
One can provide a similar construction for $\Bb^\ssl$ in $T$, giving rise to the Fourier--Mukai transform $\Psi_{\Bb^\ssl}$. Recalling the morphism $\widehat{\q}_{\Aa^\ssl}$ from \eqref{eq definition hatq_Aa}, both transforms satisfy (see \cite[(11.3.3)]{polishchuck} for instance),
\begin{equation} \label{eq FM relation}
\Psi_{\Aa^\sm} \circ R \widehat{\q}_{\Aa^\ssl, *} \cong L \Nm(\q_{\Aa^\ssl})^* \circ \Psi_{\Bb^\ssl},
\end{equation}
where the norm map $\Nm(\q_{\Aa^\ssl})$ is the dual of $\widehat{\q}_{\Aa^\ssl}$.

\begin{proposition} \label{pr FM duals}
Whenever $\q_{\Aa^\ssl}$ has ramification, the structural sheaves of the subvarieties $\left . \N_{S,\H}(\v_0, B) \right |_{\{ B \}^\ssl}$ and $\left . \P_{S,\H}(\v_0, J_0, B) \right |_{\{ B \}^\ssl}$ are dual under a relative Fourier--Mukai transform.

If $\q_{\Aa^\ssl}$ is unramified, the structural sheaf on a connected component of $\left . \P_{S,\H}(\v_0, J_0, B) \right |_{\{ B \}^\ssl}$ is Fourier--Mukai dual to $\left . \N_{S,\H}(\v_0, B) \right |_{\{ B \}^\ssl}$ equipped with the pair of line bundles obtained from pushing forward the structural sheaf under $\widehat{\q}_{\Aa^\ssl.}$ direct sum of the trivial sheaf and the pull-back of a torsion $2$ line bundle $\Ll \to \Bb^\ssl$.
\end{proposition}

\begin{proof}
We start by the ramified case. Observe that $R \widehat{\q}_{\Aa^\ssl, *} \Oo_{\Jac_{\Bb}}$ coincides with the trivial sheaf over $\N_{S,\H}(\v_0, B)|_{\{ B \}^\ssl}$ after Proposition \ref{pr q embedding on ssl}. Also, $\Psi_{\Bb^\ssl} (\Oo_{\Jac_{\Bb}})$ is the relative sky-scraper sheaf at the identity, so its pull-back under $\Nm(\q_{\Aa^\ssl})$ is precisely the trivial sheaf over $\P_{S,\H}(\v_0, J_0, B)|_{\{ B \}^\ssl}$. Then, the first statement follows easily from \eqref{eq FM relation}.

The second statement follows from the diagram \eqref{eq q factors through Z_2 on ssl} and the fact that the preimage of the Norm map has two connected components.
\end{proof}

Recalling Proposition \ref{pr FM duals}, one immediately deduces the following.

\begin{corollary} \label{co FM duals}
Whenever $\q_C$ has ramification, the structural sheaves of $\Nn_C(n,d)|_{\{ B \}^\ssl}$ and $\Pp_C(n,d,J)|_{\{ B \}^\ssl}$ are dual under a relative Fourier--Mukai transform.

If $\q_C$ is unramified, the later is Fourier--Mukai dual to the sheaf supported on $\Nn_C(n,d)|_{\{ B \}^\ssl}$ given by the direct sum of the trivial sheaf and the line bundle associated to the unramified $2:1$ cover of the family of spectral curves.
\end{corollary}

In this context, Corollaries \ref{co degeneration of N^-} and \ref{co degeneration for P^-} can be seen, respectively, as a degeneration of $\BAA$ and $\BBB$-branes, from the moduli space of sheaves on symplectic surfaces into the Hitchin system.

\end{document}